\documentclass[10pt, a4paper]{article}
\usepackage[utf8x]{inputenc}
\usepackage[toc,page]{appendix}
\usepackage{amsmath,amssymb,mathrsfs,amsthm,dsfont}
\usepackage{dsfont}
\usepackage{graphicx}
\usepackage{pst-all}
\usepackage{auto-pst-pdf}
\definecolor{vert}{HTML}{49A449}

\newtheorem{hypo}{Hypothesis}
\newtheorem{theorem}{Theorem}[section]
\newtheorem{proposition}[theorem]{Proposition}

\newtheorem{corollary}[theorem]{Corollary}
\newtheorem{lemma}[theorem]{Lemma}

\numberwithin{equation}{section}

\newcommand {\rd}{\mathrm{d}}
\newcommand{\eps}{\varepsilon}
\newcommand{\T}{{\mathbb{T}}}
\newcommand{\R}{{\mathbb{R}}}
\newcommand{\N}{{\mathbb{N}}}
\newcommand{\Z}{{\mathbb{Z}}}
\renewcommand{\P}{{\mathbb{P}}}
\newcommand{\E}{{\mathbb{E}}}

\begin{document}
\title{Stochastic acceleration in a random time-dependent potential}
\author{E. Soret\thanks{Emilie.Soret@math.univ-lille1.fr} \\and\\ S. {\ De Bi\`{e}vre}\thanks{Stephan.De-Bievre@math.univ-lille1.fr}  \\
Laboratoire Paul Painlev\'e, CNRS, UMR 8524 et UFR de Math\'ematiques \\
Universit\'e Lille 1, Sciences et Technologies\\
F-59655 Villeneuve d'Ascq Cedex, France.\\
Equipe-Projet MEPHYSTO \\
Centre de Recherche INRIA Futurs \\
Parc Scientifique de la Haute Borne, 40, avenue Halley B.P. 70478 \\
F-59658 Villeneuve d'Ascq cedex, France.
}
\date{\today}
\maketitle

%\tableofcontents

\begin{abstract}
We study the long time behaviour of the speed of a particle moving in $\R^d$ under the influence of a random time-dependent potential representing the particle's environment. The particle undergoes successive scattering events that we model with a Markov chain for which each step represents a collision. Assuming the initial velocity is large enough, we show that, with high probability, the particle's kinetic energy $E(t)$ grows as $t^{\frac25}$ when $d>5$.
\end{abstract}

%%%%%%%%%%%%%%%%%%%%%%%%%%%%%%%%%
%SECTION 1 Introduction
%%%%%%%%%%%%%%%%%%%%%%%%%%%%%%%%%

\section{Introduction}\label{s:introduction}
Our goal in this paper is to make progress on the rigorous analysis of the stochastic acceleration of a classical particle moving through a random time-dependent potential. The full problem can be described as follows. A particle moves in $\R^d$, and its position $q(t)$ obeys the following law of motion:
\begin{equation}\label{eq:newton}
\ddot q(t)= -\sum_{i} \lambda_i\nabla V(q(t)-r_i,\omega t+\phi_i),\quad q(0)=q_0,\quad \dot{q}(0)=v_0. 
\end{equation}
Here $V\in \mathcal{C}^\infty\left( \R^d,\T^m\right)$ is a real valued potential which is bounded and of compact support in its first variable in the ball of radius $\frac12$ centered at the origin.

 The frequency vector $\omega\in\R^m$ is fixed, so that the particle moves under the influence of a potential $V(q(t)-r_i,\omega t+\phi_i)$ that is quasi-periodic in time, when it is close to the scattering center $r_i$. The scattering centers $r_i\in\R^d$ are a countable and locally finite family of (random or deterministic) points that satisfies a ``finite horizon'' condition, that we shall not explicitly describe. The phases $\phi_i$ and the coupling constants $\lambda_i$ are i.i.d random variables in $\T^m$ respectively $\R$.

Such a particle undergoes successive scattering events (also refered to as collisions) when crossing one of the balls of radius $\frac12$ centered on the $r_i$, and executes a uniform straight line motion otherwise. When the potential $V$ is time-independent, the particle's kinetic energy is preserved in the scattering events and is therefore uniformly bounded in time. We are interested in the case when $V$ does depend on time, in which case the kinetic energy is expected to grow in time. This is the phenomenon known as ``stochastic acceleration''. It has been extensively studied by various authors in a variety of models (see for example \cite{GR09}, \cite{Sturrock66} and \cite{Vanden94}) and has been the subject of some controversy concerning the precise rate of growth. We refer to \cite{adblp} for further background.

In~\cite{adblp} and~\cite{aguer2010}, the above model was analysed numerically and partial arguments were given to argue that, asymptotically in time ($d\geq 2$),
$$
\E(\|\dot q(t)\|)\sim t^{1/5},\quad \E(\|q(t)\|)\sim t,
$$
where the expected value is with respect to the $(\lambda_i,\phi_i)$ and to an initial distribution of particle velocities.

In this paper, we shall consider a simplified model for the particle's motion, in which its possible recollisions with the same scatterer are ignored. Within that framework, we give a complete and rigorous analysis of the asymptotic behaviour of $\|\dot q(t)\|$ corroborating the $t^{1/5}$ law above for $d>5$ (Theorem~\ref{thm:finalonq}).

 The model is described in detail in Section~\ref{s:model}. It treats the successive scattering events as independent, leading to a Markov chain description for the particle's momentum and position at each scattering event. We therefore establish that the $t^{1/5}$ law is indeed obtained from successive random scattering events with a smooth potential. The numerics in~\cite{adblp} suggests this behaviour is not altered by possible recollisions but we do not prove this here.
 
Our work relies first of all on the analysis of the single scattering events for a high energy particle that was given in~\cite{adblp} and~\cite{aguer2010}. This yields a sufficiently sharp description of the transition probabilities of the Markov chain at high momenta to allow us to control the asymptotic behaviour of the energy of the particle in this Markov chain dynamics. For that purpose we then adapt techniques developed in~\cite{DK2009} in the context of a related problem on which we shall comment below.

The paper is organised as follow. In Section~\ref{s:model} we introduce the model that we consider and we describe the behaviour of the kinetic energy by a Markov chain where each step corresponds to a passage trough a scattering region. In Section~\ref{s:mainresult}, we state a technical result (Theorem~\ref{thm:finalresult}) for a class of Markov chains which includes the one described in Section~\ref{s:model} and we show how it implies our main result, Theorem~\ref{thm:finalonq}. In Section~\ref{s:scaling}, we show that correctly rescaled and under some technical conditions, each Markov chain of this class converges weakly to a transient Bessel process (see Theorem~\ref{thm:scaling}). This Averaging Theorem is a key element of the proof of Theorem~\ref{thm:finalresult}. Sections~\ref{s:auxprocess},~\ref{s:exittimes} and~\ref{s:finalproof} contain the three steps of the proof of Theorem~\ref{thm:finalresult}. An appendix concludes this paper with in particular the proof of Theorem~\ref{thm:scaling}.

\noindent{\bf Acknowledgments.} The authors thank B.~Aguer, M.~Rousset, T.~Simon and D.~Dereudre for helpful discussions. This work is supported in part by the Labex CEMPI (ANR-11-LABX-0007-01).

%%%%%%%%%%%%%%%%%%%%%%%%%%%%%%%%%
%SECTION 2 Markov model
%%%%%%%%%%%%%%%%%%%%%%%%%%%%%%%%%

\section{The Markov chain model}\label{s:model}
The solution $(q(t),\dot{q}(t))$ of \eqref{eq:newton} can be viewed as a stochastic process on the probability space generated by the $(\lambda_i, r_i, \phi_i)$. To each trajectory $(q(t),\dot{q}(t))$ one can associate a sequence $(t_n, v_n, b_n, r_{i_n}, \lambda_{i_n}, \phi_{i_n})_{i\in \N}$. Here $t_n$ is the instant the particle arrives at the $n$-th scattering region with incoming velocity $v_n=\dot{q}(t_n)$; $r_{i_n}$ is the $n$-th scattering center visited by the particle, $\lambda_{i_n}$ and $\phi_{i_n}$ are, respectively the associated coupling constant and phase; $b_n$ is the impact parameter (Figure~\ref{fi:scatteringevent}). More precisely, we have
$$e_n=\dfrac{v_n}{||v_n||},\quad q(t_n)=r_{i_n}-\frac{1}{2}e_n+b_n,\quad b_n\cdot v_n=0, \quad ||b_n||\leq \frac{1}{2}. $$

\begin{figure}
\centering
\includegraphics[scale=0.5]{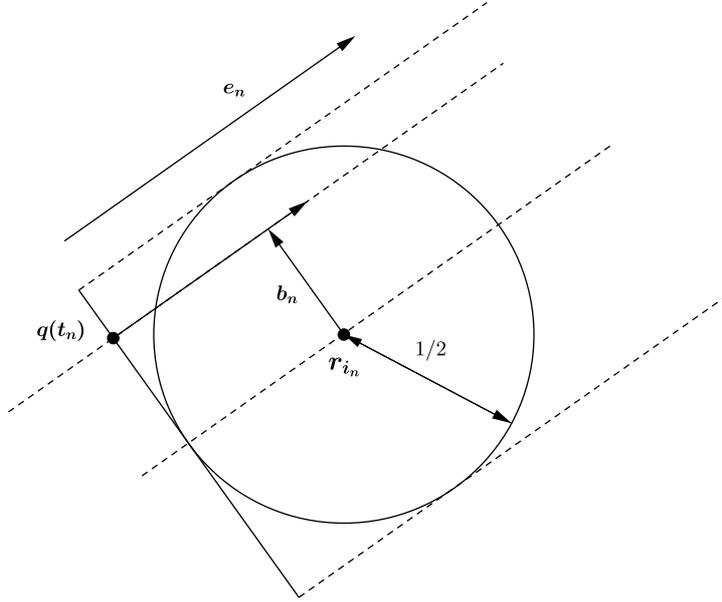}
\caption{A particle at time $t_n$ arriving with velocity $v_n$ and impact parameter $b_n$ on the $n$-th scatterer, centered at the point $r_{i_n}$.}
\label{fi:scatteringevent}
\end{figure}

The change in velocity experienced by a sufficiently fast particle at the $n$-th scattering event can be written
\begin{equation}\label{eq:vn}
v_{n+1}=v_n+R(v_n,b_n,\phi_{i_n}, \lambda_{i_n})
\end{equation}
where, for all $v\in \R^d$, $b\in \R^d$ with $v\cdot b=0$, and $(\phi,\lambda)\in \T^m\times \R,$
\begin{equation}\label{eq:momentumtransfer}
R(v,b,\phi,\lambda)=-\lambda\int_0^{+\infty}\rd t\, \nabla V(q(t),\omega t+\phi);
\end{equation}
here $q(t)$ is the unique solution of
\begin{equation}\label{eq:qonscatter}
\ddot{q}(t)=-\lambda \nabla V(q(t), \omega t +\phi), \quad q(0)=b-\frac{1}{2}\dfrac{v}{||v||}, \quad \dot{q}(0)=v. 
\end{equation}
We will always suppose the potential $V$ satisfies the following hypothesis: 
\begin{hypo}\label{hyp:firsthyp}
$V\in C^{\infty}(\R^d,\T^m)$ is bounded and of compact support in its spatial variable in the ball of radius $1/2$ at the origin. The potential $V$ and all its derivates are bounded, and we write
$$0<V_{max}:=||V||_{\infty}<+\infty.$$
Moreover, $(\omega\cdot \nabla_{\phi})V\neq 0$.
\end{hypo}

Equation \eqref{eq:vn} determines $v_{n+1}$ in term of $v_n$, $b_n$, $\phi_{i_n}$, $\lambda_{i_n}$. To determine $t_{n+1}$, $b_{n+1}$, $\lambda_{i_{n+1}}$, $\phi_{i_{n+1}}$, one would need to solve a geometric problem which consists in finding the location $r_{i_{n+1}}$ of the next scatterer visited by the particle. We shall present and study a simplified model of the dynamics in which this problem is eliminated. For that purpose note first that, once the particle leaves the $n$-th scatterer, it travels with a constant velocity $\|v_{n+1}\|$ over a distance $\ell_n$ before meeting the $n+1$-th scatterer. Hence $t_{n+1}=t_n+\frac{\ell_n}{\|v_{n+1}\|}$, where we ignored the duration of the scattering event itself. Furthermore $q_{n+1}=q_n+v_{n+1}(t_{n+1}-t_n)$ where $q_n=q(t_n)$.

Starting from this description of the dynamics and ignoring possible recollisions, we now model the solution $\left( \dot{q}(t), q(t)\right)$ of ~\eqref{eq:newton} by a coupled discrete-time Markov chain in momentum and position space as follows. Each step of the chain is associated to one scattering event. Thus, starting with a given initial velocity $v_0\gg 1$, we define iteratively the velocity $v_n$ and the time $t_n$ just before the $n$-th scattering event through the relations:
\begin{equation}\label{eq:finalrw}
\left.
\begin{array}{lll}
v_{n+1} & = & v_n+R\left(v_n, \kappa_n\right) \\
t_{n+1} & = & t_n+\dfrac{\ell}{\| v_{n+1}\|} \\
q_{n+1}&=& q_n+v_{n+1}(t_{n+1}-t_n),%
\end{array}%
\right\}
\end{equation}
where 
\begin{equation}\label{eq:kappa}
\kappa_n=(b_n, \phi_{n}, \lambda_{n}).
\end{equation}
Here the random variables $b_n$ are chosen independently at each step of the Markov chain and follow a uniform law in $B(0,\frac{1}{2})$ conditionally to $b_n\cdot v_n=0$. The variables $\lambda_{n}$ and $\phi_{n}$ are also sequences of independent random variables and identically distribued in $[-1,1]$ and $\T^m$ respectively.

Finally, note that we have added a very last simplification to this Markov chain by replacing the random variables $\ell_n$ by the mean distance $\ell$ between two scatterers successively visited by the particle. In this way the geometric problem associated to the distribution of the scatterers in the space is completely eliminated. 

This Markov chain provides a simplified but still highly non-trivial model for the original dynamical problem given in \eqref{eq:newton}. Note that the momentum change undergone by the particle during collisions is entirely encoded in the momentum transfer function $R(v,b,\phi,\lambda)$ (see \eqref{eq:momentumtransfer}) in both the original problem and the above Markov chain. The main simplifications in \eqref{eq:finalrw} come from the fact that we ignore geometric considerations (the spatial distribution of the $r_i$) as well as possible recollisions.

To state the main result of this paper, we introduce trajectories $\left( q(t)\right)_{t\in \R_+}$ where for all $n\in \N$, $q(t_n)=q_n$ is a solution of \eqref{eq:finalrw} and for all $t\in [t_n,t_{n+1}]$
\begin{equation}\label{eq:q(t)}
q(t)=q(t_n)+(t-t_n)v_{n+1}.
\end{equation}

\begin{theorem}\label{thm:finalonq}Suppose $d>5$. Then for all ${\nu}>0$ and $v_0\in \R^d$, there exist $c({\nu})>0$ depending on $\nu$ and $C(v_0,{\nu})>0$ depending on both $v_0$ and ${\nu}$ such that
\begin{equation}\label{eq:encadrementq}
\lim_{||v_0||\to +\infty}\P\left(\forall t>\frac{1}{\|v_0\|},\,  c({\nu}) t^{\frac{1}{5}-{\nu}}\leq ||\dot{q}(t)||\leq C(v_0,{\nu}) t^{\frac{1}{5}+{\nu}} \right)=1.
\end{equation}
\end{theorem}
The proof of Theorem~\ref{thm:finalonq} is given in Section~\ref{s:mainresult} where the role of the condition $d>5$ will be explained.
In order to establish this theorem, we have to analyze the behaviour of the first equation of~\eqref{eq:finalrw},
\begin{equation}\label{eq:v_n}
v_{n+1}=v_n+R(v_n,\kappa_n),
\end{equation}
for $v_n$ large. For that purpose, we need to understand the behaviour of the momentum transfer $R(v_n,\kappa_n)$ in \eqref{eq:momentumtransfer}. Using first order perturbation theory, we can write (see~\cite{adblp}),
$$R(v,\kappa)=-\frac{\lambda}{||v||}\int_{-\infty}^{+\infty} \rd y\,\nabla V\left(b+(y-\frac{1}{2})e, \dfrac{\omega t}{||v||}+\phi\right)+O\left(||v||^{-3}\right),$$
with $b\cdot v=0$. As $V$ is sufficiently smooth, we have the following expansion for $K\in \N$, $(v,\kappa)=(v, b, \phi,\lambda)\in \R^{2d}\times \T^m\times \R$  with $b\cdot v=0$:
\begin{equation*}\label{eq:devR}
R(v,\kappa)=\sum_{k=1}^K\dfrac{\alpha^{(k)}(e,\kappa)}{||v||^k}+O(||v||^{-K-1}), \qquad e=\dfrac{v}{||v||},
\end{equation*}
with
\begin{equation*}\label{eq:alpha1}
\alpha^{(1)}(e,\kappa)=-\lambda\int_{-\infty}^{+\infty}\rd y\ \nabla V\left(b+(y-\frac{1}{2})e, \phi\right).
\end{equation*}
Note that $e\cdot\alpha^{(1)}(e,\kappa)=0$. Then, if we look at the energy transfer
\begin{equation}\label{eq:DeltaE}
\Delta E(v,\kappa)=\dfrac{1}{2}\left((v+R(v,\kappa))^2-v^2\right),
\end{equation}
we have 
\begin{equation}\label{eq:devDeltaE}
\Delta E(e,\kappa)=\sum_{\ell=0}^L\dfrac{\beta^{(\ell)}(e,\kappa)}{||v||^\ell}+O(||v||^{-L-1}),
\end{equation}
where $\beta^{(0)}=e\cdot \alpha^{(1)}=0$ and $\beta^{(1)}=e\cdot \alpha^{(2)}$. Consequently, the first term in \eqref{eq:devDeltaE} is equal to $0$ and $\Delta E(e,\kappa)\sim ||v||^{-1}$.
The following theorem (see~\cite{adblp}) describes the average energy transfer during a unique collision.
\begin{theorem}\label{thm:adblp}
For all unit vectors $e\in \R^d$, $\overline{\alpha^{(1)}(e)}=0=\overline{\alpha^{(2)}(e)}$.
Moreover, for all $v\in \R^d$
$$ \overline{\Delta E(v)}=\dfrac{B}{ \| v\|^4}+O(\| v\|^{-5}), \quad \overline{(\Delta E(v))^2}=\dfrac{D^2}{\| v\|^2}+O(\| v\|^{-3}),$$
where 
$$B=\dfrac{d-3}{2}D^2$$
and
$$D^2=\dfrac{\overline{\lambda^2}}{C_d}\int_{\T^m} d\phi\,\int_{\R^{2d}}dq_0\, dq_0'\, ||q_0-q_0'||^{1-d}\partial_t\, V(q_0,\phi)\partial_t\, V(q_0',\phi)>0,$$
where $C_d$ is the volume of the sphere of radius $\frac12$ in $\R^{d-1}$. In particular, for all unit vectors $e\in \R^d$ and for all $\ell=1,2,3$,
$$\overline{\beta^{(\ell)}(e)}=0, \quad B=\overline{\beta^{(4)}(e)}, \text{ and } D^2=\overline{(\beta^{(1)}(e))^2}>0.$$
\end{theorem}
Theorem~\ref{thm:adblp} and ~\eqref{eq:DeltaE} yield
\begin{equation}\label{eq:Deltav_n}
\Delta ||v_n||^3=3\beta_n^{(1)}+\dfrac{3\left(\beta_n^{(4)}+\frac{1}{2}\left(\beta_n^{(1)}\right)^2\right)}{||v_n||^3}+O_0\left(||v_n||^{-1}\right)+O\left(||v_n||^{-4}\right).
\end{equation}
Here $\Delta ||v_n||^3=||v_{n+1}||^3-||v_n||^3$ where $(v_n)_n$ is the stochastic process defined by \eqref{eq:v_n} and $O_0\left(||v_n||^{-1}\right)$ designates a term of $O\left(||v_n||^{-1}\right)$ with zero average. Introducing
\begin{equation}\label{eq:xi}
\xi_n=\dfrac{||v_n||^3}{3D},\quad \omega_n=\dfrac{\beta_n^{(1)}}{D}, \text{ and } \gamma=\dfrac{1}{3}\left(\dfrac{B}{D^2}+\dfrac{1}{2}\right)=\frac{1}{6}(d-2)\geq -\frac{1}{6},
\end{equation}
and using \eqref{eq:Deltav_n}, we obtain the discrete Markov chain with values in $\R$
\begin{equation}\label{eq:rwxi}
\xi_{n+1}=\xi_n+\omega_n+\dfrac{\gamma}{\xi_n}+O_0\left(\xi_n^{-\frac{1}{3}}\right)+O\left(\xi_n^{-\frac{4}{3}}\right),
\end{equation}
with $\langle \omega_n\rangle=0$, $\langle \omega_n^2\rangle=1$.
To understand the behaviour of the system's kinetic energy, it remains therefore to study the Markov chain $\left( \xi_n\right)_n$, a task we turn to in the following sections. In particular, Theorem~\ref{thm:finalresult} is a technical result valid for a class of Markov chains including $\left(\xi_n\right)_n$ defined by \eqref{eq:rwxi}.

%%%%%%%%%%%%%%%%%%%%%%%%%%%%%%%%%
%SECTION 3 Main Result
%%%%%%%%%%%%%%%%%%%%%%%%%%%%%%%%%

\section{Strategy of the proof}\label{s:mainresult}
We start with a remark that explains the origin of the condition $d>5$ in Theorem~\ref{thm:finalonq}. Note that under Hypothesis~\ref{hyp:firsthyp}, a global solution of \eqref{eq:qonscatter} always exists. Nevertheless, the integral in \eqref{eq:momentumtransfer} may not converge. Indeed, it is conceivable that for given $v=\dot{q}(0)$, the solution satisfies $\|q(t)\|\leq \frac12$ for all $t>0$ large. In other words, the particle may not leave the scattering region after having entered it: it is trapped. In this case the integral in \eqref{eq:momentumtransfer} may not converge. As shown in \cite{adblp}, and as is intuitively obvious, this will not happens if $\|v\|$ is large enough (meaning $v>12|\lambda|\, \|\nabla V\|_\infty$, see \cite{adblp}). The particle will then exit the scattering region after a finite time of order $\|v\|^{-1}$. We will show below below that for $d>5$ (this means $\gamma>\frac{1}{2}$ in \eqref{eq:xi}), the Markov chain \eqref{eq:xi} is transient. This implies an initially fast particle never slows down so that there is no trapping and the chain is well defined.

We will consider a slightly more general class of Markov chains, which may be of interest on its own, and which is defined as follows. 
Let $(\omega_k)_{k\in\N}$ be a family of bounded, i.i.d. real random variables, with zero mean and whose variance equals $1$:
\begin{equation}\label{eq:iid}
\E(\omega_k)=0,\qquad \E(\omega_k^2)=1, \qquad \exists M\geq 1,\ |\omega_k|\leq M.
\end{equation}
 We will denote their common probability measure by $\mu$. Let $F:\R^+_*\times [-M, M]\to \R^+_*$ be a measurable function satisfying the following properties:
\begin{hypo}\label{hyp:one} $\exists \gamma\in\R, 0<\xi_-<\xi_+$, $\alpha>0$, $\beta>1$, such that
$F$ is continuous on $[\xi_-, +\infty[\times [-M, M]$ and, for all $\xi>\xi_+$,
\begin{equation}\label{eq:Fasym}
F(\xi, \omega)= \xi +\omega +\frac{\gamma}{\xi}+G_0(\xi,\omega)+G_1(\xi,\omega),
\end{equation}
and where the functions $G_0$ and $G_1$ are such that, for large $\xi$, 
\begin{equation}\label{eq:g0g1}
\sup_{\omega}\left\vert G_0(\xi,\omega)\right\vert =O\left( \xi^{-\alpha}\right)\text{ and }\sup_{\omega}\left\vert G_1(\xi,\omega)\right\vert =O\left( \xi^{-\beta}\right),
\end{equation}
with $\alpha>0$ and $\beta>1$. Moreover, $\E\left(G_0(\xi,\cdot)\right)=0$.
\end{hypo}
We will study the asymptotic behaviour of the Markov chains 
\begin{equation}\label{eq:Mchain}
\xi_{k+1} = F(\xi_k, \omega_k),\quad \xi_0>0.
\end{equation}
Note that the Markov chain described by \eqref{eq:rwxi} satisfies Hypothesis~\ref{hyp:one}.
The following result is the main technical ingredient for the proof of Theorem~\ref{thm:finalonq}.

\begin{theorem}\label{thm:finalresult}
Suppose $\gamma>\frac{1}{2}$. Then 
\begin{itemize}
\item[(i)]For all $0< {p} \leq 1$, for all $\nu>0$,  there exists $\xi_*>\xi_+$ such that for all $\xi_0\geq \xi_*$, we have
\begin{equation*}\label{eq:paslim}
\P\left(\forall k\in\N,\, \left(\xi_0+k^{\frac{1}{2}}\right)^{1-\nu}\leq \xi_k \leq\left(\xi_0+k^{\frac{1}{2}}\right)^{1+\nu}\right)\geq 1-{p},
\end{equation*}
\item[(ii)]For all $\nu>0$, we have
\begin{equation*}
\lim_{\xi_0\to+\infty}\P\left(\forall k\in\N,\, \left(\xi_0+k^{\frac{1}{2}}\right)^{1-\nu}\leq \xi_k \leq\left(\xi_0+k^{\frac{1}{2}}\right)^{1+\nu}\right)=1.
\end{equation*}
\end{itemize}
\end{theorem}

This asymptotic behaviour can be anticipated from the following observation. Let us consider the special case where $F$ is of the form \eqref{eq:Fasym} for all $\xi>0$ (and not only for large $\xi$) and drop the two last errors terms, i.e $F(\xi)=\xi+\omega+\frac{\gamma}{\xi}$. This is possible if $2\sqrt \gamma>M$, as is easily checked. In that case, one readily finds that
$$
\xi_{k+1}^2=\xi_k^2 + 2\gamma +\omega_k^2  +2\omega_k\left(\xi_k+ \frac{\gamma}{\xi_k}\right) +\frac{\gamma^2}{\xi_k^2},
$$
so that
$$
\E(\xi_{k+1}^2)=\E(\xi_k^2) +(2\gamma +1)+ \gamma^2 \E(\xi_k^{-2}).
$$
It follows that, for all $k\geq 2$, 
\begin{equation}\label{eq:borne}
 \xi_0^2 +(2\gamma +1)k\leq \E(\xi_{k}^2)\leq \xi_0^2 +(2\gamma +1+\frac14\frac{\gamma}{(1 -\frac{M}{2\sqrt\gamma})^2})k.
\end{equation}
It shows that, indeed, $\E(\xi_k^2)$ behaves as $k$ in this simple case. Of course, this information on the second moment of $\xi_k$ does not imply the statement of the Theorem ~\ref{thm:finalresult}, even in this case. Conversely, the statement of the Theorem ~\ref{thm:finalresult} does not allow to draw conclusions on the moments of $\xi_k$, since we have no control on the trajectories on a set of small probability.

\begin{figure}[t]
\psset{xunit=1cm,yunit=0.9cm}
\begin{pspicture}(19,9)
\rput[bl](3,0){\includegraphics[width=19cm,height=9cm]{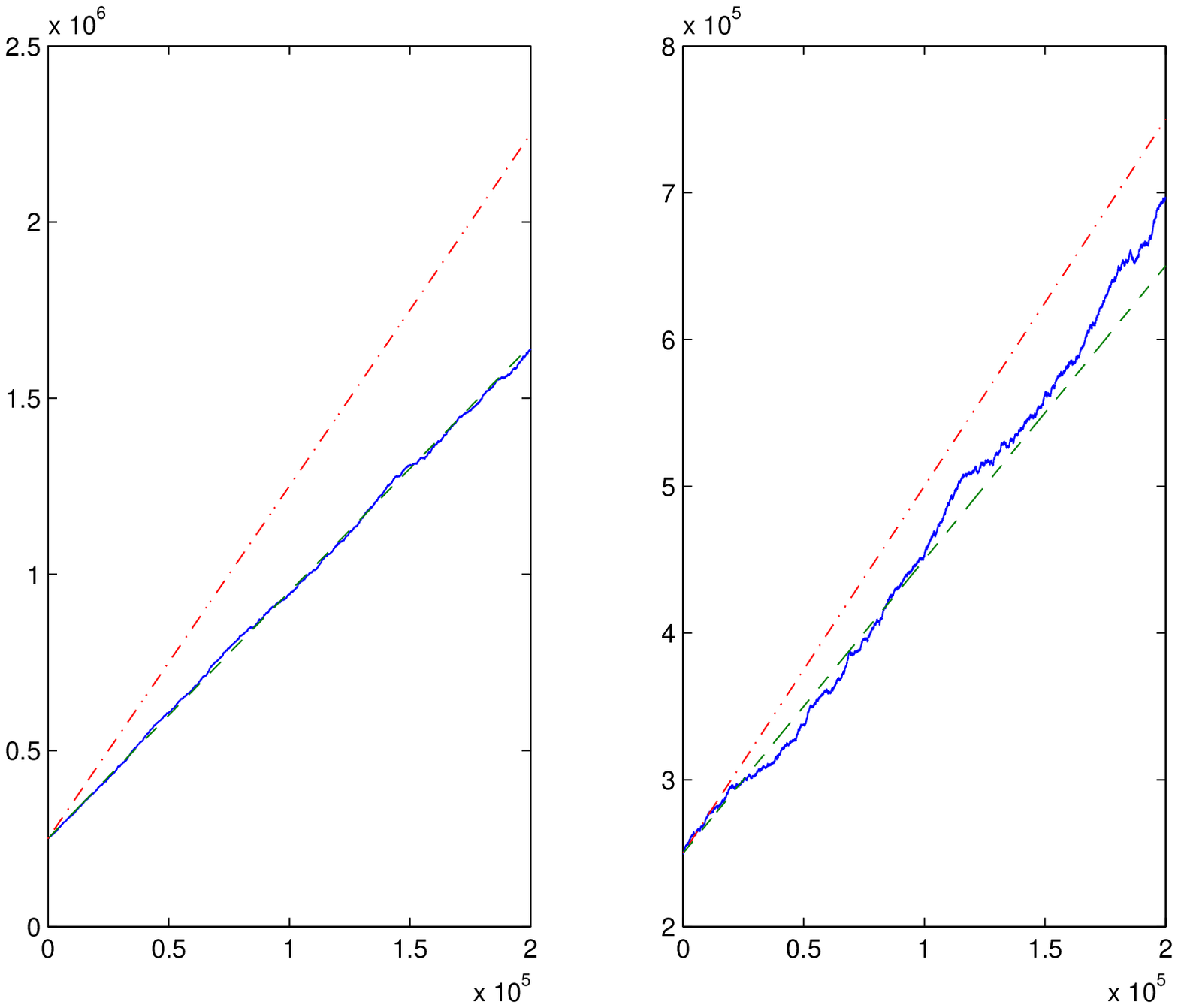}}
\rput[bl](6.2,8.15){\psline[linecolor=red](0.13,0)\psline[linecolor=red](0.23,0)(0.235,0)\psline[linecolor=red](0.335,0)(0.465,0)}
\rput[bl](16,3){\psline[linecolor=vert](0.13,0)\psline[linecolor=vert](0.23,0)(0.36,0)\psline[linecolor=vert](0.46,0)(0.59,0)}
\rput[bl](8.2,1.5){$\gamma=3$}
\rput[bl](16.6,1.5){$\gamma=\frac{1}{2}$}
\begin{scriptsize}
\rput[l](7,8){$k\to\xi_0^2 +(2\gamma +1+\frac14\frac{\gamma}{(1 -\frac{M}{2\sqrt\gamma})^2})k $}
\rput[l](16.8,3){$k\to \xi_0^2+(2\gamma+1)k$}
\end{scriptsize}
\end{pspicture}
\caption{Illustration of \eqref{eq:borne} with different values of $\gamma$. Mean of $500$ simulations of a Markov chain $(\xi_k)_k$.}
\label{fi:espsquare}
\end{figure}

Another way to anticipate the asymptotic behaviour of $\xi_k$ is to notice that the Markov chain 
$$
\xi_{k+1}=\xi_k+\omega_k+ \frac{\gamma}{\xi_k}
$$
can be thought of as a time discretized version of the stochastic differential equation satisfied by a Bessel process $R_t$ of dimension $2\gamma+1$:
$$
\rd R_t = \rd B_t +\frac{\gamma}{R_t}\rd t,
$$
where $B_t$ is a standard Brownian motion and $\gamma>-\frac{1}{2}$. It is of course well known (see~\cite{RevuzYor}) that $R_t\sim \sqrt{ (2\gamma +1) t}$ when $\gamma>-\frac{1}{2}$. In Section~\ref{s:scaling}, a rigorous version of this observation constitutes the first step of  the proof of Theorem~\ref{thm:finalresult}. Indeed, we introduce a family of rescaled processes $R_t^\epsilon$ and then show that the $R_t^\epsilon$ converge, as $\epsilon\to 0$, to a Bessel process with $R(0)=1$ (Theorem~\ref{thm:scaling}). We note that the transience and recurence of various time and space discretized versions of the Bessel process are discussed in~\cite{A11} and~\cite{CFR09} but no results on their asymptotic behaviour are obtained there.

Observe that in Hypothesis~\ref{hyp:one} no assumption is made on the behaviour of the chain when $\xi<\xi_-$. Such information is unavailable in the application we have in mind as we already indicated, and it is therefore important to see what can be said without it. Clearly, one cannot hope to obtain general results valid for all $\gamma$, without such additional information. Indeed, if $\gamma$ is too small, the trajectories will reach the region $]0, \xi_+]$ with probability one, and the asymptotic behaviour of the chain will then depend crucially on the behaviour of $F$ in that region. This can be seen for example when $\gamma=0$, and $F(\xi,\omega)=\xi+\omega$, for all $\xi\geq \xi_-$. In that case,  we are dealing with an ordinary random walk for $\xi>\xi_-$, which is recurrent. If then $\xi_->M$ and $F(\xi,\omega)=0$ for all $\xi<\xi_-$,  it is clear that, with probability $1$, $\lim_{k\to+\infty}\xi_k=0$ (and $\E(\xi_k)\to0$). On the other hand, if $F(\xi, \omega)=|\xi+\omega|$, $\forall \xi,\omega$, then $\E(\xi_k)\sim k^{1/2}$ and $\limsup_k \xi_k=+\infty$, with probability $1$. In short, when $\gamma$ is small, the chain is recurrent and one needs a ``non-trapping'' condition of the trajectories in the region $[0,\xi_+]$ to ensure the asymptotic behaviour of $\xi_k$ is still of the form $k^{1/2}$. 

Once we have Theorem~\ref{thm:finalresult}, we can show Theorem~\ref{thm:finalonq}.

\begin{proof}[Proof of Theorem~\ref{thm:finalonq}]
Theorem~\ref{thm:finalresult} \textit{ii)} and \eqref{eq:xi} yield that for all $\nu>0$
\begin{equation}\label{eq:q-1}
\lim_{\|v_0\|\to+\infty}\P\left(\forall n\geq 0;\, (3D)^{-\frac13}\left(\frac{\|v_0\|^3}{3D}+n^{\frac12}\right)^{-\frac{1+\nu}{3}}\leq \|\dot{q}(t_n)\|^{-1}\leq(3D)^{-\frac13}\left(\frac{\|v_0\|^3}{3D}+n^{\frac12}\right)^{-\frac{1-\nu}{3}}\right)=1.
\end{equation}
Furthermore, by \eqref{eq:finalrw} we have, for all $n\geq 1$, 
\begin{equation}\label{eq:tnsum}
t_n=\ell \sum_{j=1}^{n} \|v_j\|^{-1},\quad t_0=0.
\end{equation}
Combining \eqref{eq:q-1} and \eqref{eq:tnsum}, straightforward estimates show that for all $\nu>0$, the following bounds on $t_n$ hold,
\begin{equation}\label{eq:tn}
\lim_{\|v_0\|\to+\infty}\P\left(\forall n\geq 2;\, c_1(\nu)t_n^{\frac{6}{5+\nu}}\leq n\leq C_1(v_0,\nu)t_{n-1}^{\frac{6}{5-\nu}}\right)=1.
\end{equation}
Here $c_1(\nu)>0$ and $C_1(v_0,\nu)>0$ are two positive constants depending only on $\nu$ and $(v_0, \nu)$ respectively. This implies, by part \textit{ii)} of Theorem~\ref{thm:finalresult}, that for all ${\nu}>0$
\begin{equation}\label{eq:q-2}
\lim_{\|v_0\|\to+\infty}\P\left(\forall n\geq 2;\, c_2(\nu)t_n^{\frac{1-\nu}{5+\nu}}\leq \|\dot{q}(t_n)\|\leq C_2(v_0,\nu)t_{n-1}^{\frac{1+\nu}{5-\nu}}\right)=1.
\end{equation}
Then, as for all $t\in (t_{n-1},t_n]$, $\dot{q}(t)=\dot{q}(t_n)$ (see \eqref{eq:q(t)}), it follows from \eqref{eq:q-2} that
\begin{equation}\label{eq:q-3}
\lim_{\|v_0\|\to+\infty}\P\left(\forall t>\frac{\ell}{\|v_1\|} ;\, c(\nu)t^{\frac{1-\nu}{5+\nu}}\leq \|\dot{q}(t)\|\leq C(v_0,\nu)t^{\frac{1+\nu}{5-\nu}}\right)=1.
\end{equation}
Using \eqref{eq:q-1}, this result is easily extended to all $t>1/\|v_0\|$.

%It remains to interpolate to all $t\geq 0$. For all $t_{n}\geq t> t_{n-1}$ we have $\dot{q}(t)=\dot{q}(t_n)$, so that~\eqref{eq:q-2} implies the result for all $t\geq t_1$. For the case $0<t\leq t_1$ we have $\|\dot{q}(t)\|=\|v_1\|$ and $t\geq \frac{1}{\| v_0\|}$ then the upper bound of Theorem~\ref{thm:finalresult} \textit{ii)} yields that for all $\tilde{\nu}>0$
%$$\|v_1\|\leq (3D)^{\frac13}\left(\frac{\|v_0\|^3}{3D}+1\right)^{\frac{1+\tilde{\nu}}{3}}\times \|v_0\|^{\frac{1}{5}+\tilde{\nu}}\left(\dfrac{1}{\|v_0\|}\right)^{\frac{1}{5}+\tilde{\nu}} \leq C_3(v_0,\tilde{\nu})\, t^{\frac{5}{6}+\tilde{\nu}},$$
%The lower bound is easy to obtain, always by Theorem~\ref{thm:finalresult} \textit{ii)} and as $t_1\ll 1$, we have
%$$\|v_1\|\geq (3D)^{\frac{1}{3}}\left(\frac{\|v_0\|^3}{3D}+1\right)^{\frac{1-\tilde{\nu}}{3}}\times 1\geq C_4(v_0,\tilde{\nu})\, t^{\frac{1}{5}-\tilde{\nu}}.$$
%Here $C_3(v_0,\tilde{\nu})$ and $C_4(v_0,\tilde{\nu})$ are two positive constants only depending on $v_0$ and $\tilde{\nu}$.

\end{proof}

The rest of this paper is devoted to the proof of Theorem ~\ref{thm:finalresult}. The strategy is the following. We will consider, in Section~\ref{s:scaling}, a family of Markov processes $R^{\eps}_n=\eps \xi_n$, indexed by $\eps=\xi_0^{-1}$. We show that after an appropriate rescaling of the time variable, the limit of this new family as $\eps\to 0$ is a Bessel process of dimension $2\gamma+1=\dfrac{d+1}{3}$ when $\gamma>\frac{1}{2}$ ($d>5$ in the initial problem). This yields Theorem~\ref{thm:scaling}. The proof of this averaging theorem is given in Appendix~\ref{s:appendix}.
In Section~\ref{s:auxprocess}, implementing a strategy developed in \cite{DK2009} for a similar problem, we define an auxiliary process $\eta_\ell\in \Z$ and corresponding stopping times $\tau_\ell$ such that, roughly, $\xi_{\tau_\ell}\sim 2^{\eta_\ell}$ (see Figure~\ref{fi:eta}). In other words, the increments of the process $(\eta_\ell)_\ell$ are $\pm 1$, and $\Delta \tau_\ell=\tau_{\ell+1}-\tau_{\ell}$ is the time the process $(\xi_n)_n$ needs to double or half its value.
In Section~\ref{s:exittimes}, we use Theorem~\ref{thm:scaling}, properties of the Bessel process and the Porte-Manteau Lemma to show that, provided $\gamma>\frac{1}{2}$ and $\eta_0$ is large enough, $(\eta_\ell)_\ell$ is a submartingale. We then control $\Delta \tau_\ell$. Basically, we show (Proposition~\ref{prop:jumpproba}) that there exists $\mu>0$ such that
$$\eta_\ell\sim \mu \ell+\eta_0, \text{ and } \Delta \tau_\ell\sim 2^{2\eta_\ell}. $$
In Section ~\ref{s:finalproof}, we use the results of Sections ~\ref{s:auxprocess} and ~\ref{s:exittimes} to conclude the proof of Theorem~\ref{thm:finalresult}.

We end this section with a further comment on~\cite{DK2009}. The authors of that paper study a similar model, in which however the force does not derive from a potential. In other words, it is not irrotational. In that case they show that, provided $\|\dot q(0)\|$ is large enough, and for $d\geq 4$,
$$
\|\dot q(t)\| \sim t^{1/3},\qquad \|q(t)\|\sim t^{4/3},
$$
with high probability. Note that the energy growth is faster here than when the force derives from a potential as in our case: it grows as $t^{2/3}$ as compared to $t^{2/5}$ in the latter situation. This faster growth allows the authors of~\cite{DK2009} to show the spatial trajectories of the particles do not self-intersect, so that recollisions do in fact occur only with very low probability. This in turn allows them to control the growth of $\|q(t)\|$. The situation under study in this paper is very different. As argued and shown numerically in~\cite{adblp}, the slower growth of the energy when the force does derive from a potential leads the particle to turn on a short time scale, so that self-intersections of the trajectory do occur and the growth of $\|q(t)\|$, as $t$, is slower than the power $t^{6/5}$ one could naively expect. In fact, the numerics of~\cite{adblp} indicates $\|q(t)\|\sim t$. We will come back to this aspect of the problem in a further publication. 
%%%%%%%%%%%%%%%%%%%%%%%%%%%%%%%%%
%SECTION 4 Scaling limit
%%%%%%%%%%%%%%%%%%%%%%%%%%%%%%%%%

\section{A scaling limit}\label{s:scaling}
Let $\eps_*>0$, to be fixed later. We introduce $\eps=\xi_0^{-1}$, and define, for $\eps<\eps_*$,
$$
R^\eps_n:=\eps{\xi_n}.
$$
Note that $R^\eps_0=1$, independently of $\eps$. It then follow from \eqref{eq:Fasym} that $R^\eps_n$ satisfies
\begin{equation*}
R_{n+1}^\epsilon= G(\epsilon, R_n^\epsilon, \omega_n)
\end{equation*}
where, for $x> R_+^\eps=\eps\xi_+$
\begin{equation*}
G(\epsilon, x, \omega)=x+\epsilon\omega+\epsilon^2\gamma x^{-1}+\eps^{\alpha+1}G^\eps_0\left(x,\omega\right)+\eps^{\beta+1}G^\eps_1\left(x,\omega\right),
\end{equation*}
where $G^\eps_0$ and $G^\eps_1$ are such that
$$\sup_{\omega \in \Omega}\left\vert G^\eps_0(x,\omega)\right\vert =O\left( x^{-\alpha}\right)\text{ and }\sup_{\omega \in \Omega}\left\vert G^\eps_1(x,\omega)\right\vert =O\left( x^{-\beta}\right),$$
with $\alpha>0$ and $\beta>1$. Moreover, $\E\left(G^\eps_0(x,\cdot)\right)=0$.

We then construct a continuous time process by linear interpolation, as follows.  For $n\in\N$, $t_n=n\eps^2$ and for $t\in [t_n,t_{n+1}]$,
\begin{equation*}
R^{\epsilon}(t)=\frac{t_{n+1}-t}{\epsilon^2}R_n^\epsilon +\frac{t-t_n}{\epsilon^2}R_{n+1}^\epsilon.
\end{equation*}

\begin{theorem} \label{thm:scaling} Fix $T\in\R^+_*$. If $\gamma> 1/2$, the processes $(R^\epsilon_t)_{t\in[0,T]}$ converge weakly, as $\epsilon\to0$, to the Bessel process of dimension $2\gamma+1$, and with initial condition $1$. 
\end{theorem}
The condition on $\gamma$ guarantees that the limiting Bessel process is transient, does not explode in finite time and does not reach zero. This is an important element of the proof which is given in Appendix ~\ref{s:appendix}.

In addition, we will need the following result.
\begin{lemma}\label{lem:exittimeBessel}
Let $\gamma>-\frac12$, and let $R$ be a Bessel process of dimension $2\gamma+1$ with $R(0)=1$. Let, for $a_-<1<a_+$, 
$$
T_{a_-,a_+}=\inf\{t\geq 0 \mid R(t)\not \in ]a_-, a_+[\},\quad T_{a_-}=\inf\{t\geq 0 \mid R(t)<a_-\}, \quad T_{a_+}=\inf\{t\geq 0 \mid R(t)> a_+\}.
$$
(i) Then, for all $T\geq 0$,
\begin{equation*}\label{eq:Besselexit}
0<\P(T_{a_-, a_+}>T)<1.
\end{equation*}
(ii) If in addition $\gamma>\frac12$, 
\begin{equation*}\label{eq:Besseltransient}
\P(T_{a_-}>T_{a_+})=\dfrac{a_-^{1-2\gamma}-1}{a_-^{1-2\gamma}-a_+^{1-2\gamma}}
\end{equation*}
\end{lemma}
In order not to break the flow of the main argument, the proofs of Theorem~\ref{thm:scaling} and Lemma~\ref{lem:exittimeBessel} are given in Appendix~\ref{s:appendix}.

%%%%%%%%%%%%%%%%%%%%%%%%%%%%%%%%%
%SECTION 5 Auxiliary process
%%%%%%%%%%%%%%%%%%%%%%%%%%%%%%%%%

\section{An auxiliary process}\label{s:auxprocess}
Let $L>0$, $\eta\in \Z$ and define the intervals $J_\eta=[2^\eta-L, 2^\eta+L]$. We consider the subset $\mathcal{N}_L:= \bigcup_{\Z}J_\eta$ of $\R_*^+$ and we will study how the Markov chain $\left( \xi_k\right)_k$ visits successively $\mathcal{N}_L$ by introducing an auxiliary process $\left(\eta_\ell\in \Z\right)_\ell$ and corresponding stopping times $\tau_\ell$, so that $\xi_{\tau_\ell}\in J_{\eta_\ell}$, (see Figure~\ref{fi:eta}).
We start with a technical remark. Note that in Hypothesis~\ref{hyp:one}, $\xi_+$ can always be replaced by a larger value. It turns out to be convenient to work under the following further condition on $\xi_+$: 
$\xi_+\geq  \frac{|\gamma|}{M}$.
Under this hypothesis, one easily checks that
\begin{equation}\label{eq:stepsizecontrol}
\forall \xi_k> \xi_+,  \qquad \xi_{k+1}\in \left(\xi_k-C_{M,\gamma}, \xi_k+C_{M,\gamma}\right),
\end{equation}
where $C_{M,\gamma}=2M+C\dfrac{M^\alpha}{|\gamma|^\alpha}+C\dfrac{M^\beta}{|\gamma|^\beta}$, $C>0$.
This expresses the rather obvious fact that, for large enough $\xi_k$, the step size of the random walk is small compared to $\xi_k$.

Let us now define the process $(\eta_\ell)_{\ell\in\N}$ precisely. First, set
\begin{equation}\label{eq:etaplus}
\eta_+=\min\{\eta\in\N\ |\ 2^\eta > 2\max\{\xi_+, C_{M,\gamma}\}\}> 1,
\end{equation}
where the last inequality follows from the observation that $M\geq 1$ (See \eqref{eq:iid}). In view of \eqref{eq:etaplus}, one can choose $L$ satisfying $C_{M,\gamma}<L<2^{\eta_+-1}$, from which it follows that, for all $\eta,\eta'\geq\eta_+$, $\eta\not=\eta'$, we have $J_\eta\cap J_{\eta'}=\emptyset.$ Note that, in view of \eqref{eq:stepsizecontrol}, the process $\xi_k$ cannot jump across one of these intervals without visiting it. In this way, for all $\ell$, $\eta_{\ell+1}=\eta_\ell\pm 1$,as we will see.

We are now in  a position to define the process $\eta_\ell$, and the associated stopping times $\tau_\ell$ recursively, as follows. We restrict ourselves to initial conditions $\xi_0$ for which there exists an integer $\eta_0$ so that $\xi_0\in J_{\eta_0}$, with $\eta_0>\eta_+$. Note that if $\xi_0$ is not in such an interval, by Lemma~\ref{lem:exittimebound} we can control the time that the procces spend before entering in $J_{\eta_0}$. Then, define $\tau_0=0$ and
\begin{equation*}
\tau_1:=\inf \{k> \tau_0\ |\ \xi_k\in J_{\eta_0-1}\cup J_{\eta_0+1}\}.
\end{equation*}
We define
$$
\eta_1=\eta_0+1,\ \mathrm{if}\ \xi_{\tau_1}\in J_{\eta_0+1},\ \mathrm{and}\  \eta_1=\eta_0-1,\ \mathrm{if}\ \xi_{\tau_1}\in J_{\eta_0-1}.
$$   
We then proceed recursively. Suppose that, for some $\ell\in\N$,  $\tau_0, \eta_0, \tau_1, \eta_1, \dots \tau_\ell, \eta_\ell$ have been defined, with $\xi_{\tau_k}\in J_{\eta_k}$, for all $0\leq k\leq \ell$. If $\eta_\ell=\eta_+$, we define $\tau_{\ell+1}=\tau_\ell$  and $\eta_{\ell+1}=\eta_\ell$. Otherwise we define
\begin{align*}
\tau_{\ell+1} &= \inf\{k> \tau_\ell\ |\ \xi_k\in J_{\eta_\ell-1}\cup J_{\eta_\ell+1} \},\\
\eta_{\ell+1}=\eta_\ell+1 ,\ & \mbox{if}\ \xi_{\tau_{\ell+1}}\in J_{\eta_\ell+1},\,\,\mbox{and}\,\, \eta_{\ell+1}=\eta_\ell-1 ,\, \mbox{if}\ \xi_{\tau_{\ell+1}}\in J_{\eta_\ell-1}.
\end{align*}

\begin{figure}
\centering
\includegraphics[scale=0.9]{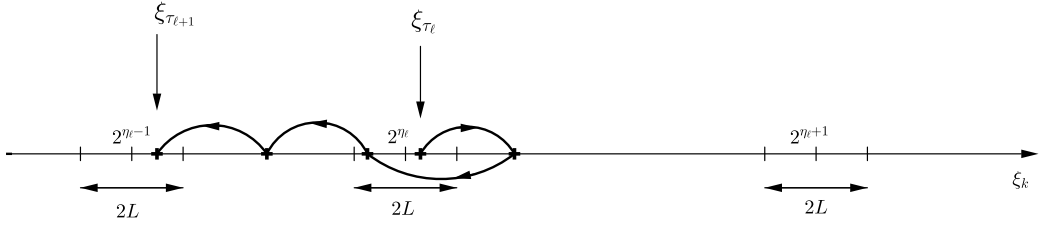}
\caption{$(\xi_k)_k$ visiting $\mathcal{N}_L$. Here, $\eta_{\ell+1}=\eta_\ell-1$.}
\label{fi:eta}
\end{figure}

We will show in this section that the process $(\eta_\ell)_\ell$ is asymptotically a submartingale, with high probability, and that $\eta_\ell\sim \mu \ell$, for some $\mu>0$ (see Proposition~\ref{prop:jumpproba} (ii)). In Section~\ref{s:exittimes}, we will combine this result with estimates on the dwell times $\tau_{\ell+1}-\tau_\ell$ between successive visits of the original process $\xi_k$ to $\mathcal{N}_L$, which we show to be of order $(2^{\eta_\ell})^2\sim 2^{2\mu\ell}$, to conclude that
$$
\tau_\ell\sim 2^{2\mu\ell}, \quad \mathrm{and\ hence}\quad \xi_{\tau_\ell}\sim \sqrt{\tau_\ell}.
$$
(See Proposition~\ref{prop:dwelltimes} (ii)\&(iv) for a precise statement.) It will then remain, in Section~\ref{s:finalproof}, to interpolate between the stopping times $\tau_\ell$ to obtain Theorem~\ref{thm:finalresult}.

Note that the sequence $(\tau_\ell)_\ell$ is increasing, and we have the following dichotomy: either the sequence $(\tau_\ell)_\ell$ is strictly increasing, $\lim_{\ell\to+\infty}\tau_\ell = +\infty$, and $\forall \ell\in\N$, $\eta_\ell> \eta_+$, or  $\exists L_*\in \N$ and $T_*>0$ so that $\tau_\ell=T_*$  and $\eta_\ell=\eta_+$, forall $\ell\geq L_*$. 

There is no reason to think the process $\eta_\ell$ is still a Markov process, specifically $(\eta_\ell)_\ell$ describe the behaviour of $(\xi_k)_k$ on interval:
\begin{align*}
\P\left(\eta_{\ell+1}=\eta\pm 1|\eta_\ell=\eta, \eta_{\ell-1}=\eta\pm 1\right)&=\P\left(\xi_{\tau_{\ell+1}}\in J_{\eta\pm 1}| \xi_{\tau_\ell}\in J_\eta, \, \xi_{\tau_{\ell-1}}\in J_{\eta\pm 1}\right)\\
& \neq \P\left(\xi_{\tau_{\ell+1}}\in J_{\eta\pm 1}| \xi_{\tau_\ell} \in J_\eta \right).
\end{align*}
Actually, it depends on if $\xi_{\tau_\ell}$ is rather on the left than on the right of $J_{\eta_\ell}$.

To control its asymptotic behaviour, we will  show it is, with high probability, a submartingale, if $\eta_0$ is sufficiently large, and control its jump probabilities $\P(\eta_{\ell+1}=\eta_\ell\pm1\ |\ \eta_\ell,\dots,\eta_0)$. (See Proposition~\ref{prop:jumpproba} (i).)  We note that the transience of the chain $\left(\xi_k\right)_k$ is essential in the arguments of this section; it is, as we shall see, ensured by the condition that $\gamma>\frac12$. The main properties of the process $\eta_\ell$ are summarized in the following proposition.
\begin{proposition}\label{prop:jumpproba} 
\begin{itemize}
\item[(i)]  Suppose $\gamma> \frac{1}{2}$. For all $\delta>0$ there exists $\tilde\eta>\eta_+$ such that for all $\ell\in \N^*$ and for almost all $\eta_0,\eta_1, \dots, \eta_{\ell-1}\geq\eta_+$, $ \eta_\ell>\tilde\eta$, we have
\begin{equation}\label{eq:jumpproba}
|\P\left(\eta_{\ell+1}=\eta_\ell\pm1 | \eta_\ell, \dots,\eta_0\right)- p_\pm|<\delta,
\end{equation}
where $p_+=\frac{2^{2\gamma-1}-1}{2^{2\gamma-1}-2^{1-2\gamma}}>\frac12$ and $p_-=1-p_+$. 
\item[(ii)] For all $0<p\leq 1$ and for all $\delta>0$, there exists $\eta_*>\eta_+$ so that for all $\eta_0\geq\eta_*$
\begin{equation*}\label{eq:coneetal}
\P\left(\left|\eta_\ell-\mu\ell-\eta_0\right|\leq\delta\left(\ell+\eta_0\right),\forall \ell\in\N \right)\geq 1-p,
\end{equation*}
where $\mu=2p_+-1>0$.
\item[(iii)]For all $0<p\leq 1$, there exists $\eta_*>\eta_+$ so that for all $\eta_0\geq\eta_*$,
\begin{equation*}
\P\left(\forall \ell\in\N,\eta_\ell\geq \dfrac{\eta_0}{2}\right)\geq 1-p.
\end{equation*}
\end{itemize}
\end{proposition}

We start with two preliminary observations. First, in what follows our notation will not distinguish between on the one hand the random variable $\P\left(A| \eta_\ell, \dots,\eta_0\right)$, viewed as a function on the underlying probability space or on $\N^{\ell+1}$, and on the other hend the values it takes at  points in $\N^{\ell+1}$, also denoted by $(\eta_\ell, \dots,\eta_0)\in\N^{\ell+1}$. Second, we will often make use of the following useful property of the process $\xi_k$, which is a consequence of its Markovian nature:
\begin{equation}\label{eq:markovineq}
\inf_{\xi_{k-1}\in I} \P(A\in \mathcal F_k^+|\xi_{k-1})\leq \P(A\in \mathcal F_k^+|\xi_{k-1}\in I, C\in \mathcal F_{k-2})\leq \sup_{\xi_{k-1}\in I} \P(A\in \mathcal F_k^+|\xi_{k-1}),
\end{equation}
where $I$ is an interval, $\mathcal F_{k-2}$ is the sigma-algebra generated by the $\xi_{k'}, 0\leq k'\leq k-2$
and $\mathcal F_k^+$ the sigma-algebra generated by the $\xi_{k'},  k'\geq k$. 
\begin{proof}
(i) Let $\eta_0,\eta_1,\dots, \eta_\ell>\eta_+$.  We then have 
\begin{eqnarray}\label{eq:etaell+1}
 \P\left(\eta_{\ell+1}=\eta_\ell+1 | \eta_\ell, \dots,\eta_0\right) &
 =&\sum_{0<i_1\leq\dots\leq i_\ell}\P\left(\eta_{\ell+1}=\eta_\ell+1| \eta_\ell, \dots,\eta_0;\ \tau_\ell=i_\ell, \dots, \tau_1=i_1\right)\nonumber\\
 &\ & \qquad\qquad\qquad\qquad\times\P\left( \tau_\ell=i_\ell, \dots, \tau_1=i_1  \mid \eta_\ell, \dots,\eta_0\right).
\end{eqnarray}
Here and in what follows, the values of $\eta_0,\dots,\eta_\ell$ and of the multi-indices  $i_j$ are restricted to values for which the set on which we condition has non-zero probability. Introducing, for all $i\in \N$ and for all $\eta>\eta_+$,  
$$
\tau_{+,i}(\eta)=\inf\{k\geq 0| \xi_{i+k}> 2^{\eta+1}-L\}\quad\mathrm{and}\quad \tau_{-,i}(\eta)=\inf\{k\geq 0| \xi_{i+k}<2^{\eta-1}+L \},
$$ 
we can write, for all $\eta_0,\eta_1,\dots, \eta_\ell>\eta_+$, and for all $0<i_1<\dots< i_\ell$, 
\begin{eqnarray} 
\P\left(\eta_{\ell+1}=\eta_\ell+1 |\eta_\ell, \dots,\eta_0;\ \tau_\ell=i_\ell, \dots, \tau_1=i_1\right)&\ &\nonumber\\
&\ &\hskip-6cm =\P\left(\tau_{+,i_\ell}(\eta_\ell)<\tau_{-,i_\ell}(\eta_\ell)| \eta_\ell, \dots,\eta_0;\ \tau_\ell=i_\ell, \dots, \tau_1=i_1\right)\nonumber\\
&\ &\hskip-6cm =\P\left(\tau_{+,i_\ell}(\eta_\ell)<\tau_{-,i_\ell}(\eta_\ell)| \xi_{i_\ell}\in J_{\eta_\ell}, \dots,\xi_{i_1}\in J_{\eta_1}, \xi_0\in J_{\eta_0};\ \tau_\ell=i_\ell, \dots, \tau_1=i_1\right).\label{eq:partition}
\end{eqnarray}
It then follows from \eqref{eq:markovineq} and the homogeneity of the process $\xi_k$ that
\begin{align}
\inf_{\xi_{0}\in J_{\eta_\ell}}\P(\tau_{+,0}(\eta_\ell)&<\tau_{-,0}(\eta_\ell)| \xi_{0})=\inf_{\xi_{i_\ell}\in J_{\eta_\ell}}\P\left(\tau_{+,i_\ell}(\eta_\ell)<\tau_{-,i_\ell}(\eta_\ell)| \xi_{i_\ell}\right)\nonumber\\
& \leq \P\left(\tau_{+,i_\ell}(\eta_\ell)<\tau_{-,i_\ell}(\eta_\ell)| \xi_{i_\ell}\in J_{\eta_\ell}, \dots,\xi_{i_1}\in J_{\eta_1}, \xi_0\in J_{\eta_0};\ \tau_\ell=i_\ell, \dots, \tau_1=i_1\right)\nonumber\\
& \leq \sup_{\xi_{i_\ell}\in J_{\eta_\ell}} \P\left(\tau_{+,i_\ell}(\eta_\ell)<\tau_{-,i_\ell}(\eta_\ell)| \xi_{i_\ell}\right)=\sup_{\xi_{0}\in J_{\eta_\ell}}\P\left(\tau_{+,0}(\eta_\ell)<\tau_{-,0}(\eta_\ell)| \xi_{0}\right).\nonumber
\end{align}
 Inserting this into \eqref{eq:partition} and using the result in \eqref{eq:etaell+1} finally yields
\begin{equation}\label{eq:encadrement}
\inf_{\xi_{0}\in J_{\eta_\ell}}\P\left(\tau_{+,0}(\eta_\ell)<\tau_{-,0}(\eta_\ell)| \xi_{0}\right)\leq \P\left(\eta_{\ell+1}=\eta_\ell+1 | \eta_\ell, \dots,\eta_0\right) \leq \sup_{\xi_{0}\in J_{\eta_\ell}}\P\left(\tau_{+,0}(\eta_\ell)<\tau_{-,0}(\eta_\ell)| \xi_{0}\right).
\end{equation}
We will now use the Porte-Manteau Theorem again to conclude the argument. For ease of notation, we shall write $\tau_{\pm}(\eta)=\tau_{\pm,0}(\eta)$ in what follows. Let $\eta>\eta_+$ and $\xi\in J_{\eta}$. We consider the set 
$E_\xi(\eta)$ defined as follows:
\begin{displaymath}
E_\xi(\eta)=\{ \tau_+(\eta)<\tau_-(\eta);\, \xi_0=\xi \},
\end{displaymath}
so that
\begin{equation}\label{eq:inf<infEn}
\inf_{\xi\in J_\eta}\P\left(\tau_+(\eta)<\tau_-(\eta)|\xi_0=\xi\right)= \inf_{\xi\in J_\eta}\P\left(E_\xi(\eta)\right),\qquad \sup_{\xi\in J_\eta}\P\left(\tau_+(\eta)<\tau_-(\eta)|\xi_0=\xi\right)= \sup_{\xi\in J_\eta}\P\left(E_\xi(\eta)\right).
\end{equation}
%[[CHECK how to put this cleanly. Since $\xi_0$ is fixed, there is no difference with the conditional expectation, really. One way to see it is to consider $E$ as a subset of trajectories rather than $\Omega$, probably. ]]\\
Noting that 
\begin{displaymath}
E_\xi(\eta)=\{ \forall k<\tau_+(\eta),\, \xi_k>2^{\eta-1}+L; \, \xi_0=\xi \}, 
\end{displaymath}
one sees, with the notation of Section~\ref{s:scaling} ($R^\eps_k=\frac{\xi_k}{\xi_0}$, and $\eps=\xi_0^{-1}$), that, provided $\xi\in J_\eta$,
\begin{displaymath}
E_\xi(\eta)\supset\{\forall k<\tau_+(\eta),\, R^\eps_k>\frac{1}{2}\sigma_-(\eta);\, R^\eps_0=1\},
\end{displaymath}
where $\sigma_-(\eta)=\dfrac{2^\eta+2L}{2^\eta-L}$. Note that $\sigma_-$  is a decreasing function of its argument which tends to $1$ as $\eta\to+\infty$. Let $\eta_*>\eta_+$, to be chosen later, as a function of $\delta$ in \eqref{eq:jumpproba}. Let $\eta>\eta_*$; it then follows that
\begin{displaymath}
E_\xi(\eta)\supset\{ \forall k<\tau_+(\eta),\, R^\eps_k>\frac{1}{2}\sigma_-(\eta);\, R^\eps_0=1\}\supset\{ \forall k<\tau_+(\eta),\, R^\eps_k>\dfrac{1}{2}\sigma_-(\eta_*);\, R^\eps_0=1\}.
\end{displaymath}
In order to apply the Porte-Manteau Theorem, we need to replace the stopping time $\tau_+(\eta)$ of $\xi_k$ by an appropriately chosen stopping time of the continuous time process $R^\epsilon(t)$ introduced in Section~\ref{s:scaling}. We will proceed in two steps. First we replace $\tau_+(\eta)$ by a stopping time $\tau^\eps_+$ for the discrete time process $R^\eps_k$, which is defined as follows: 
$$
\tau^\eps_+=\inf\{k\geq 0| R^\eps_k>2\sigma_+(\eta_*)\},
$$ 
where $\sigma_+(\eta_*)=\dfrac{2^{\eta_*}-\frac{L}{2}}{2^{\eta_*}-L}$  is also decreasing and tends to $1$ as $\eta_*\to +\infty$. One checks that $\tau^\eps_+\geq \tau_+(\eta)$, for all $\eta>\eta_*$, so that for all $\eta>\eta_*$, and for all $\xi\in J_\eta$, 
\begin{align*}
E_\xi(\eta)\supset\{ \forall k<\tau_+(\eta),\, R^\eps_k>\dfrac{1}{2}\sigma_-(\eta_*);\, R^\eps_0=1\}\supset\{\forall k<\tau^\eps_+,\, R^\eps_k>\dfrac{1}{2}\sigma_-(\eta_*);\, R^\eps_0=1\}.
\end{align*}
We next consider two stopping times $T_+^\eps$ and $T_-^\eps$  for the continuous time processes $\left(R^\eps(t)\right)_t$, defined as follows
\begin{displaymath}
T^\eps_+=\inf\{t\geq 0| R^\eps(t)=2\sigma_+(\eta_*)\}\quad \mbox{and}\quad T_-^\eps=\inf\{t\geq 0| R^\eps(t)=\frac{1}{2}\sigma_-(\eta_*)\}.
\end{displaymath}
It then follows from the definition  of $\left(R^{\eps}(t)\right)_t$ by linear interpolation of the $(R^\eps_k)_k$ between the times $t_k=k\eps^2$, and the fact that $\tau^\eps_+$ is an integer that $(\tau^\eps_+-1)\eps^2<T_+^\eps<\tau_+^\eps \eps^2$, so that $\frac{1}{2}\sigma_-(\eta_*)<R^\eps\left((\tau^\eps_+-1)\eps^2\right)<2\sigma_+(\eta_*)$ and $R^\eps\left(\tau^\eps_+\eps^2\right)>2\sigma_+(\eta_*)$. 
\begin{figure}
\centering
\includegraphics[scale=0.8]{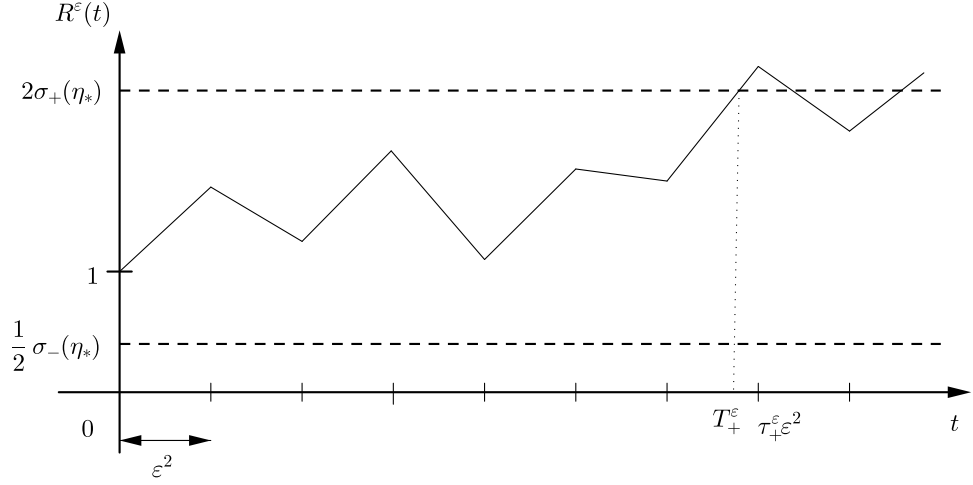}
\caption{The linear interpolation of the $\left(R^\eps_k\right)_k$ between $t_k$ and $t_{k+1}$ yields that $T_{\pm}^\eps$ happens before $\eps^2\tau_{\pm}^\eps$.}
\label{fi:raisonning}
\end{figure}
 Hence
\begin{displaymath}
\{\forall t<\tau^\eps_+\eps^2,\, R^\eps(t)>\dfrac{1}{2}\sigma_-(\eta_*);\, R^\epsilon(0)=1\}\supset \{\forall t<T^\eps_+,\, R^\eps(t)>\dfrac{1}{2}\sigma_-(\eta_*);\, R^\epsilon(0)=1\}.
\end{displaymath}
Finally, we may conclude that, for all $\eta>\eta_*$ and for all $\xi\in J_\eta$, with $\epsilon=\xi^{-1}$
\begin{equation}\label{eq:Ensupset}
\P\left(E_\xi(\eta)\right)\geq \P\left(\forall t<T^\eps_+,\, R^\eps(t)>\frac{1}{2}\sigma_-(\eta_*);\, R^\epsilon(0)=1\right).
\end{equation}
We will now apply the Porte-Manteau Theorem to get a lower bound on the right hand side of this inequality. For that purpose, we first remark that, for all $\eps$ and $\eta_*>\eta_+$, 
\begin{displaymath}
\P\left(\forall t<T^\eps_+,\, R^\eps(t)>\dfrac{1}{2}\sigma_-(\eta_*);\, R^\eps(0)=1\right)=\P\left(\forall t<T^\eps_+,\, R^\eps(t)>\dfrac{1}{2}\sigma_-(\eta_*);\, R^\eps(0)\in ]\frac{5}{6},\frac{7}{6}[\right),
\end{displaymath}
because $\P\left(R^\eps(0)\in]\frac{5}{6},1[\cup]1,\frac{7}{6}\right)=0$.
The set $\{\varphi\in\mathcal{C}([0,T],\R_+)| \forall t<{T_{a_+}},\, \varphi(t)>\frac{1}{2}\sigma_-(\eta_*);\, \varphi(0)\in]\frac{5}{6},\frac{7}{6}[\}$ where ${T_{a_+}}=\inf\{t\geq 0| \varphi(t)=a_+\}$, is open. Hence the Porte-Manteau Theorem together with Theorem~\ref{thm:scaling} and Lemma~\ref{lem:exittimeBessel} imply
\begin{equation}\label{eq:portemanteau}
\liminf_{\eps\to 0}\P\left(\forall t<T_+^\eps,\, R^\eps(t)>\frac{1}{2}\sigma_-(\eta_*)\right)\geq \P\left(\forall t<T_{a_+},\, R(t)>a_-\right)=\P(T_{a_+}<T_{a_-})=\dfrac{a_-^{1-2\gamma}-1}{a_-^{1-2\gamma}-a_+^{1-2\gamma}},
\end{equation}
where we use the notation of Lemma~\ref{lem:exittimeBessel} with $a_-=\frac{1}{2}\sigma_-(\eta_*)$, $a_+=2\sigma_+(\eta_*)$ and where $\left(R(t)\right)_t$ is a Bessel process of dimension $2\gamma+1$ and initial condition $R(0)=1$. Since $\sigma_\pm(\eta_*)\to 1$ when $\eta_*\to+\infty$, there exist $\eta_*$ large enough, depending only on $\delta$ and $L$, so that, 
$$
\P(T_{a_+}<T_{a_-})\geq p_+-\frac{\delta}{2}, \quad \mathrm{where}\quad p_+=\frac{2^{2\gamma-1}-1}{2^{2\gamma-1}-2^{1-2\gamma}}>\frac12,
$$
since $\gamma>\frac12$.
It then follows from \eqref{eq:portemanteau} that there exists $\tilde\eps$ so that  
\begin{equation}\nonumber
\inf_{\eps<\tilde\eps}\P\left(\forall t<T_+^\eps,\, R^\eps(t)>\frac{1}{2}\sigma_-(\eta_*)\right)\geq p_+-\delta,
\end{equation}
Combining this with \eqref{eq:encadrement}, \eqref{eq:inf<infEn} and \eqref{eq:Ensupset}, we obtain
\begin{equation*}\label{eq:liminffinal}
\inf_{\eta_\ell>\tilde\eta}\P\left(\eta_{\ell+1}=\eta_\ell+1|\eta_\ell,\dots\eta_0\right)\geq p_+-\delta,
\end{equation*}
where $\tilde\eta=\max\{ \eta_*, \log_2(\tilde\eps^{-1}+L)\}$.  
This is the desired lower bound on the jump probability of the process $\eta_\ell$. 

To control the upper bound in \eqref{eq:encadrement}, we proceed in the same manner. 
First, for all $\xi\in J_\eta$, $\eps=\xi^{-1}$, 
\begin{align*}
E_\xi(\eta)&=\{ \forall k<\tau_+(\eta),\, \xi_k>2^{\eta-1}+L;\, \xi_0=\xi\}
\subset \{ \forall k<\tau_+(\eta),\, R_k^\eps>\frac{2^{\eta-1}+L}{2^\eta+L};\, R^\eps_0=1\}\\
&\subset \{\forall k<\tau_+(\eta),\, R_k^\eps>\frac{1}{2};\, R^\eps_0=1\}.
\end{align*}
Now, let $\tilde\eta_*>\eta_+$, to be chosen later, and let $\eta>\tilde\eta_*$.  Consider the stopping time $\tilde{\tau}_+^{\eps}=\inf\{k\geq 0| \, R_k^{\eps}>2\tilde\sigma_+(\tilde\eta_*)\}$ where $\tilde{\sigma}_+(\tilde\eta_*)=\dfrac{2^{\tilde\eta_*}-\frac{L}{2}}{2^{\tilde\eta_*}+L}$. Note that $\tilde{\sigma}_+$ is increasing and converges to $1$ when $\tilde\eta_*\to +\infty$. One readily checks  that $\tilde{\tau}_+^\eps\leq \tau_+(\eta)$ and hence
\begin{align*}
\{ \forall k<\tau_+(\eta),\, R_k^\eps>\frac{1}{2};\, R^\eps_0=1\}&\subset \{ \forall k<\tilde{\tau}_+^\eps,\, R_k^\eps>\frac{1}{2};\, R^\eps_0=1\}\\
&\subset \{ \forall t<\tilde{T}_+^\eps,\, R^\eps(t)>\frac{1}{2};\, R^\eps(0)=1\},
\end{align*}
where $\tilde{T}_+^\eps=\inf\{t\geq 0| \, R^\eps(t)=2\tilde\sigma_+(\tilde\eta_*)\}$ and $(\tilde{\tau}_+^\eps-1)\eps^2\leq \tilde{T}_+^\eps<\tilde{\tau}_+^\eps\eps^2$.  Finally, we have
\begin{equation}\label{eq:Ensubset}
E_\xi(\eta)\subset\{ \forall t<\tilde{T}_+^\eps,\, R^\eps(t)>\frac{1}{2};\, R^\eps(0)=1\}.
\end{equation}
Set $\tilde a_+=2\tilde\sigma_+(\tilde \eta_*)$. Now, we can again use the Porte-Manteau Theorem and Theorem~\ref{thm:scaling}, because the set $\{\varphi\in \mathcal{C}\left([0,T],\R_+\right)|\, \forall t\leq {T}_{\tilde a_+}, \, \varphi(t)>\frac{1}{2};\, \varphi(0)=1\}$ where ${T}_{\tilde a_+}=\inf\{t\geq 0|\, \varphi(0)=1, \varphi(t)=\tilde a_+\}$ is closed. This leads to
\begin{equation}\label{eq:limsupEn}
\limsup_{\eps\to 0}\P\left(\forall t\leq \tilde{T}^\eps_+,\ R^\eps(t)>\frac{1}{2}\right)\leq \P\left(\forall t\geq {T}_{\tilde a_+},\, R(t)>\frac{1}{2}\right)
\end{equation}
where $\left(R(t)\right)_t$ is as before a Bessel process of dimension $2\gamma+1$ and initial condition $R(0)=1$. Defining $\tilde a_-=\frac12$, and using Lemma~\ref{lem:exittimeBessel}, we have
\begin{equation}\label{eq:TR2}
\P\left(\forall t\leq {T}_{\tilde a_+},\, R(t)>\frac{1}{2}\right)=\P\left({T}_{\tilde a_-}>{T}_{\tilde a_+}\right)=\frac{2^{2\gamma-1}-1}{2^{2\gamma-1}-\tilde a_+^{1-2\gamma}}.
\end{equation}
It follows from \eqref{eq:limsupEn} and \eqref{eq:TR2} that there exists $\tilde\eps$ depending on $\delta$ so that
$$
\sup_{\eps<\tilde\eps} \P\left(\forall t\leq \tilde{T}^\eps_+,\ R^\eps(t)>\frac{1}{2}\right)\leq p_++\delta.
$$
Combining  this with \eqref{eq:encadrement} and \eqref{eq:Ensubset}, we see there exists $\tilde\eta>\eta_+$ so that
\begin{equation*}\label{eq:limsupfinal}
\sup_{\eta_\ell>\tilde\eta}\P\left(\eta_{\ell+1}=\eta_\ell+1| \, \eta_\ell,\dots, \eta_0\right)\leq p_++\delta,
\end{equation*}
which is the desired upper bound.

(ii) Let $0<\delta<\mu$ and $0<p\leq 1$. We first write down the Doob decomposition (see \cite{EthierKurtz}) of $\eta_\ell$ explicitly:
\begin{equation*}
\eta_\ell = \eta_0 + M_\ell + A_\ell,
\end{equation*}
where
\begin{equation*}
M_\ell =\sum_{j=1}^{\ell} \left(\eta_j-\E\left(\eta_j|\eta_{j-1},\dots,\eta_0\right)\right),\quad\mathrm{and}\quad A_\ell = \sum_{j=1}^\ell \left( \E\left(\eta_j|\eta_{j-1}, \dots, \eta_0\right)-\eta_{j-1}\right).
\end{equation*}
As is well known, and easily checked, $M_\ell$ is a martingale with respect to the natural filtration induced by the process $\eta_\ell$, a fact we will use below. Now,
$$
\mid\eta_\ell -(\eta_0+\mu\ell)\mid \leq \mid M_\ell\mid  + \sum_{j=1}^\ell \mid\left[\left( \E\left(\eta_j|\eta_{j-1}, \dots, \eta_0\right)-\eta_{j-1}\right)-\mu\right]\mid.
$$
It then follows from part (i) of the Lemma that, for all $\delta>0$, there exists $\tilde\eta>\eta_+$ so that, 
\begin{equation}\label{eq:driftestimate}
\forall \ell\in \N, \left (\eta_{\ell-1},\dots, \eta_0>\tilde \eta\Rightarrow \mid\eta_\ell -(\eta_0+\mu\ell)\mid \leq \mid M_\ell\mid  + \frac\delta2 \ell\right).
\end{equation}
Now, for any $L>0$, define
\begin{equation*}\label{eq:FL}
F_L=\{|M_\ell|\leq \frac\delta2\ell, \forall \ell\geq L\}.
\end{equation*}
Then, on $F_L$, and provided $\eta_0>\tilde \eta+L>\eta_++L$, so that $\eta_{L-1},\dots, \eta_0>\tilde \eta$, one has
$$
\mid\eta_L -(\eta_0+\mu L)\mid \leq \delta L,
$$
so that in particular $\eta_L>\eta_0>\tilde \eta +L$. This in turn implies that $\eta_j>\tilde \eta$, for all $0\leq j\leq 2L$. 
\begin{figure}
\centering
\includegraphics[scale=0.5]{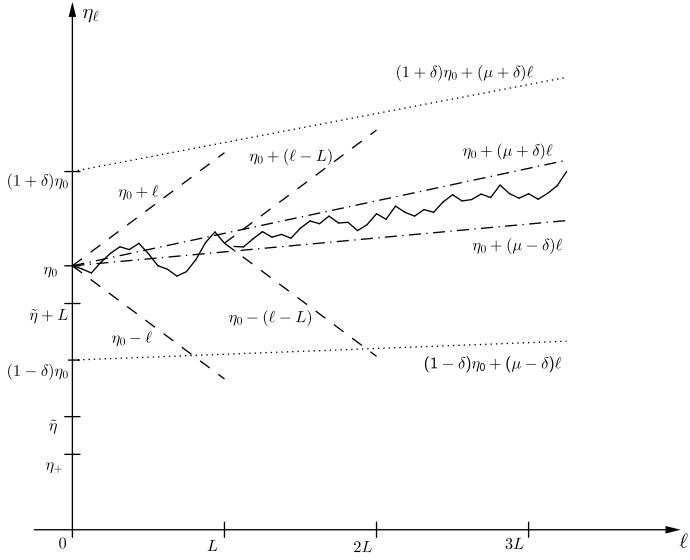}
\caption{A typical trajectory of $\eta_\ell$, on $F_L$, as defined in \eqref{eq:FL}.} 
\end{figure}
We can therefore apply \eqref{eq:driftestimate} for all $L\leq \ell \leq 2L$ to conclude that on $F_L$, and provided $\eta_0>\tilde \eta+L>\eta_++L$, one has
\begin{equation}\label{eq:driftestimatebis}
\mid\eta_\ell -(\eta_0+\mu\ell)\mid \leq \mid M_\ell\mid  + \frac\delta2 \ell\leq\delta\ell.
\end{equation}
Proceeding recursively, one then concludes that \eqref{eq:driftestimatebis} holds on $F_L$, for all $L\leq \ell$. For $0\leq \ell \leq L$, one has from \eqref{eq:driftestimate} that 
$$
\mid\eta_\ell -(\eta_0+\mu\ell)\mid \leq \mid M_\ell\mid  + \frac\delta2\ell\leq 2L+\delta\ell.
$$
Hence, if we choose $\eta_0>\frac{2L}{\delta}$, we can conclude that, 
$$
\forall\, 0\leq \ell \leq L;\,
\mid\eta_\ell -(\eta_0+\mu\ell)\mid \leq  \delta(\eta_0+\ell).
$$
From this, and \eqref{eq:driftestimatebis}, we conclude that, for all $\delta>0$ and all $L>0$, if 
$$
\eta_0>\eta_*=\max\{\tilde \eta+L, \frac{2L}{\delta}\},
$$
then
\begin{equation}\label{eq:drift3}
\P\left(\forall \ell\in\N, \mid\eta_\ell -(\eta_0+\mu\ell)\mid \leq  \delta(\eta_0+\ell)\right)\geq \P(F_L).
\end{equation}
It remains to show that, given $\delta$ and $p$, there exists $L$ so that 
$$
\P(F_L)\geq 1-p
$$
to conclude the proof. For that purpose, let us introduce the quadratic variation of $M_\ell$,
$$
D_\ell^2=\sum_{i=1}^{\ell} (M_{i}-M_{i-1})^2.
$$
The Burkholder inequality (see \cite{Karatzas}) then says that, for all $\ell\in\N$, exists a constant $C>0$
$$
\E\left((\max_{0\leq i\leq \ell}|M_i|)^4\right)\leq C \E(D_\ell^4).
$$
The definition of the $\eta_\ell$ and of the martingale $M_\ell$ immediately imply that, with probability one, $|M_\ell-M_{\ell-1}|\leq 2$, for all $\ell$. This implies immediately that $\E(D_\ell^4)\leq 4\ell^2$. Hence, by the Tchebychev inequality,
$$
\P\left(|M_\ell|> \frac\delta2\ell\right)\leq \P\left(\max_{0\leq i\leq \ell}|M_i|>\frac\delta2\ell\right)\leq
\frac{2^4}{\delta^4\ell^4}C4\ell^2=\frac{\tilde{C}}{\delta^4\ell^2},
$$
where $\tilde{C}$ is a numerical constant. It then follows that
$$
\P(F_L)\geq 1-\sum_{\ell\geq L} \P(|M_\ell|> \frac\delta2\ell)\geq 1-\frac{\tilde{C}}{\delta^4L}.
$$
Choosing $L=\frac{\tilde{C}}{\delta^4p}$, the result now follows from \eqref{eq:drift3}.

(iii) This is an immediate consequence of (ii).
\end{proof}

%%%%%%%%%%%%%%%%%%%%%%%%%%%%%%%%%
%SECTION 6 Dwell times
%%%%%%%%%%%%%%%%%%%%%%%%%%%%%%%%%

\section{Estimates on the dwell times $\tau_\ell-\tau_{\ell-1}$}\label{s:exittimes}
As explained in the introduction of Section~\ref{s:auxprocess}, having obtained the asymptotic behaviour of $\eta_\ell$, we now need to control the stopping times $\tau_\ell$ and show that with high probability they behave, roughly, as $\tau_\ell\sim \xi_{\tau_\ell}^2\sim 2^{2\eta_\ell}$. We turn to this task in this section, the main result of which is stated in Proposition~\ref{prop:dwelltimes} (ii)\&(iv). For that purpose, we will first estimate the dwell times $\tau_\ell-\tau_{\ell-1}$ (Proposition~\ref{prop:dwelltimes} (i)\&(iii)). Roughly speaking, this is the time the process needs to move from $\xi_{\tau_{\ell-1}}$ to either $2\xi_{\tau_{\ell-1}}$ or to $\frac12\xi_{\tau_{\ell-1}}$. As we will see in Lemma~\ref{lem:exittimebound}, the latter can be estimated from above and from below using Theorem~\ref{thm:scaling}, together with the Porte-Manteau Theorem and Lemma~\ref{lem:exittimeBessel} (ii), a task we now turn to.

Let us define, for all $n_0\in\N$, for all $b_-< 1<b_+$, and for all $r>0$, the stopping time
\begin{equation}\label{eq:Kn0}
K_{n_0}=\inf\{ k\in\N\ |\ \xi_{n_0+k}\not\in ]b_- r, b_+r[\}.
\end{equation}

If $\xi_{n_0}\not\in  ]b_- r, b_+r[$, $K_{n_0}=0$. Otherwise, $K_{n_0}> 0$: $n_0+K_{n_0}-1$ is then the last instant that the process is still inside the interval $]b_-r, b_+r[$. The following lemma gives the bounds on $K_{n_0}$ that we shall be needing.

\begin{lemma}\label{lem:exittimebound} Suppose Hypothesis~\ref{hyp:one} holds and that $\gamma\geq 1/2$. (i) There exists $\xi_*>\xi_+$ and $0<q_-<1$, so that, for all $m\in\N_*$, for all $n_0\in\N$,
\begin{equation*}\label{eq:Kupperbd}
\sup_{r\geq \xi_*}\sup_{\xi_{n_0}\in]b_-r, b_+ r[}\P\left(m r^2< K_{n_0}\ |\ \xi_{n_0}\right)\leq q_-^m\, <1.
\end{equation*}
 (ii) Let $b'_-, b_+'$ be such that
$b_-<b'_-< 1< b'_+< b_+$. Then there exists $\xi_*>\xi_+$ and $0<q_+<1$ so that, for all $n_0\in\N$,
\begin{equation}\label{eq:Klowerbd}
\sup_{r\geq \xi_*}\sup_{\xi_{n_0}\in]b_-'r, b_+^, r[}\P\left( K_{n_0}\leq r^2\ |\ \xi_{n_0}\right)\leq q_+\, <1.
\end{equation}
\end{lemma}

\begin{proof} (i) We first treat the case with $m=1$. Let $r>0$, $\xi_{n_0}\in]b_-r, b_+r[$. The homogeneity of the Markov chain implies it is enough to consider $n_0=0$. Consider the set
$$
\{K_{0}> r^2\}=\{\forall\, 0\leq k\leq r^2, \xi_k\in ]b_-r, b_+r[\}=\{\forall t_k\leq \eps^2r^2, R^\eps(t_k) \in ]\eps b_-r, \eps b_+r[\}.
$$
where we used the notation of Section~\ref{s:scaling}. Since $(b_+r)^{-1}< \eps < (b_-r)^{-1}$, it follows that 
$$
\{K_{0}> r^2\}\subset \{\forall t_k\leq \eps^2r^2, 
R^\eps(t_k) \in]\frac{b_-}{b_+}, \frac{b_+}{b_-}[\}\subset
\{\forall t_k\leq 2T, 
R^\eps(t_k) \in]\frac{b_-}{b_+}, \frac{b_+}{b_-}[\},
$$
where $T=\frac{1}{2b_+^2}$. Now choose $r>(b_-\sqrt{T})^{-1}$ so that $\eps^2<T$. Then, if $\tilde K$ satisfies $t_{\tilde K}\leq 2T<t_{\tilde K+1}$, we have $T<t_{\tilde K}$.  Since $R^\eps(t)$ is constructed by linear interpolation between the $R^\eps(t_k)$, we can then conclude that
$$\{\forall t_k\leq 2T, 
R^\eps(t_k) \in]\frac{b_-}{b_+}, \frac{b_+}{b_-}[\}\subset
\{\forall t\leq T, 
R^\eps(t) \in]\frac{b_-}{b_+}, \frac{b_+}{b_-}[\}\subset
\{\forall t\leq T, 
R^\eps(t) \in[\frac{b_-}{b_+}, \frac{b_+}{b_-}]\},
$$
so that
\begin{equation}\label{eq:probF}
\forall r>(b_-\sqrt{T})^{-1}, \forall \xi_0\in ]b_-r, b_+r[,\quad\P(\{K_{0}> r^2\})\leq \P(\forall t\leq T, 
R^\eps(t) \in[\frac{b_-}{b_+}, \frac{b_+}{b_-}]).
\end{equation}
The set $\{\varphi\in C([0,T],\R_*)\ |\ \forall t\in [0,T], \varphi(t)\in [\frac{b_-}{b_+}, \frac{b_+}{b_-}]\}$ is closed, so we can apply the Porte-Manteau Theorem,
together with Theorem~\ref{thm:scaling} to conclude that
\begin{equation}\label{eq:portemanteau1}
\exists \eps_*>0, \forall \eps<\eps_*,\quad\P(\{\forall t\leq T, 
R^\eps(t) \in[\frac{b_-}{b_+}, \frac{b_+}{b_-}]\})\leq q_-:=\frac12(1+\tilde q_-),
\end{equation}
where
$
\tilde q_-:=\P(\forall t\in [0,T], R_t\in [\frac{b_-}{b_+}, \frac{b_+}{b_-}]\}).
$
By Lemma~\ref{lem:exittimeBessel}, $\tilde q_-<1$ so that $q_-<1$.
It then follows from \eqref{eq:probF} and \eqref{eq:portemanteau1} that
$$
\forall r>\xi_*=\max\{(b_-\sqrt T)^{-1}, (b_-\eps_*)^{-1}\}, \forall \xi_0\in ]b_-r, b_+r[,\quad
P( \{K_{0}> r^2\})\leq q_-<1.
$$
This proves \eqref{eq:Kupperbd} for $m=1$.

It remains to show the case $m>1$. This will follow from the Markov property of the chain, as follows. We write $\Delta=]b_-r , b_+r[$. Let us introduce $K_*=\lfloor {r^2}\rfloor$, where $\lfloor \cdot \rfloor$ denotes the integer part. First note that
\begin{eqnarray}
\P(K_{n_0}> m K_*\ |\ \xi_{n_0})&=&\P(\xi_{n_0+1}\in\Delta, \dots, \xi_{n_0+m K_*}\in\Delta\ |\ \xi_{n_0})\nonumber\\
&=& \Pi_{s=0}^{m-1}\Pi_{i=1}^{K_*} \int_{\xi_{n_0+sK_*+i}\in\Delta}\P(\rd \xi_{n_0+sK_*+i}\ |\  \xi_{n_0+sK_*+i-1})\nonumber\\
&=& \Pi_{s=0}^{m-1} \int_{\xi_{n_0+(s+1)K_*}\in\Delta} L_{K_*}(\rd \xi_{n_0+(s+1)K_*}, \xi_{n_0+ sK_*}),\nonumber\\
\label{eq:markovprop}
\end{eqnarray} 
where 
\begin{eqnarray*}
L_{K_*}(A,y)&:=& \int_{\xi_{n_0+ (s+1)K_*}\in A}\int_{\xi_{n_0+(s+1)K_*-1}\in\Delta}\dots\int_{\xi_{n_0+sK_*+1}\in\Delta}\nonumber\\
&\ &\quad \Pi_{i=2}^{K_*}\P(\rd \xi_{n_0+sK_*+i}\ |\ \xi_{n_0+sK_*+i-1})\P(\rd \xi_{n_0+sK_*+1}\ |\ \xi_{n_0+sK_*}=y)\nonumber\\
&=& \P(\xi_{n_0+(s+1)K_*}\in A, \xi_{n_0+sK_*+i}\in\Delta, 1\leq i< K_*\ |\ \xi_{n_0+sK_*}=y),\nonumber
\end{eqnarray*}
which does in fact not depend on $s$, nor on $n_0$, because the Markov chain is homogeneous. Now remark that, when $A=\Delta$, and $y\in\Delta$, one finds
$$
L_{K_*}(\Delta, y) = \P(K_{n_0}> K_*\ |\ \xi_{n_0}=y).
$$
It then follows from \eqref{eq:markovprop} and from \eqref{eq:Kupperbd} for $m=1$ that, for all $m\in \N_*$,
$$
\P(K_{n_0}> m K_*\ |\ \xi_{n_0})\leq q_-^m.
$$
This completes the proof of (i).
 
(ii) The argument is analogous to the first part of (i). Again, because of the homogeneity of the chain, it is enough to prove the result for $n_0=0$. Let $\xi_0\in  ]b_-'r, b_+'r[\subset ]b_-r, b_+r[$. We then have 
\begin{eqnarray*}
 \{ K_{0}\leq r^2\}&
=&\{ \exists k\leq r^2, \xi_{k}\not\in\Delta\}\\
&\subset& \{ \exists t_k\leq \frac{1}{{b_-'}^2}, R^\eps(t_k)\not\in
]\frac{b_-}{b_-'}, \frac{b_+}{b'_+}[\}\\
&\subset& \{ \exists t\in [0,T], R^\eps(t)\not\in
]\frac{b_-}{b_-'}, \frac{b_+}{b'_+}[\}
\end{eqnarray*}
where we set $T=\frac{1}{{b_-'}^2}$. The set $\{\varphi\in C([0,T],\R_*)\ |\ \exists t\in [0,T], \varphi(t)\not\in ]\frac{b_-}{b_-'}, \frac{b_+}{b'_+}[\}$ is closed, so we can apply the Porte-Manteau Theorem, together with Theorem~\ref{thm:scaling} and Lemma~\ref{lem:exittimeBessel} to obtain \eqref{eq:Klowerbd} with
$$
q_+=\frac12(1+\tilde q_+),\quad \tilde q_+=\P(\exists t\in [0,T], R_t\not\in ]\frac{b_-}{b_-'}, \frac{b_+}{b'_+}[\})<1.
$$
\end{proof}
To state the main result of this section, we introduce ``good'' sets where the dwell times are suitably controlled and that we will show to be of high probability. Let $\eta_0>\eta_+, \delta>0$ be given, as well as two increasing sequences $(k^\pm_\ell)$ of positive integers, with $0<k_\ell^-\leq \ell$. Define furthermore the sequence
\begin{equation}\label{eq:al}
a_\ell=2^{2\left[(1-\delta)\eta_0+(\mu-\delta)(l-1-k_\ell^-)\right]}2^{-\delta k^-_\ell}.
\end{equation} 
 %Also, can you replace $\ell-1$ in exponent by $\ell$?]]
Then we introduce
\begin{align*}
G_1&=\{\forall \ell\in\N, |\eta_\ell-\mu\ell-\eta_0|\leq \delta (\ell+\eta_0)\}=\cap_\ell G_1(\ell),\\
G_2&=\{\forall\ell\in\N_*, \tau_\ell-\tau_{\ell-1} \leq k_{\ell-1}^+2^{2\eta_{\ell-1}}\}=\cap_\ell G_2(\ell),\\
G_3&=\{\forall \ell\in\N_*, \exists \ell-k_\ell^-\leq k\leq \ell, \tau_k-\tau_{k-1}\geq a_\ell\}=\cap_\ell G_3(\ell)\quad \mathrm{and}\quad G=G_1\cap G_2\cap G_3.
\end{align*}
If we set $\delta=0$, and $k_\ell^-=0$, $k_{\ell-1}^+=1$, and use \eqref{eq:al}, then one can easily check that on $G$, $\tau_\ell \sim 2^{2\mu \ell}$ and $\xi_{\tau_\ell}\sim 2^{\mu \ell}$, this mean that $\xi_{\tau_\ell}\sim \sqrt{\tau_\ell}$ which is the power law we are trying to establish. But in that case, we cannot hope to prove a suitable lower bound on $\P(G)$. To do so, we need to make the set $G$ a little bigger, by taking $\delta>0$ and choosing suitable growing sequences $k_\ell^\pm$. This will allow us to show $\P(G)$ is close to $1$ in the following proposition, using Proposition~\ref{prop:jumpproba} and Lemma~\ref{lem:exittimebound}, and at the same time to get suitable bounds on $\tau_\ell$ in function of $2^{2\eta_\ell}\sim \xi_{\tau_\ell}^2$.
\begin{proposition}\label{prop:dwelltimes} (i) $\forall 0\leq p <1$, and for all $\delta>0$, $\exists \tilde\eta>\eta_+$ so that $\forall \eta_0>\tilde \eta$ and for all sequences $(k^+_\ell)_{\ell\in\N}$, we have
\begin{equation}\label{eq:dwelltimeub}
\P(G_1\cap G_2)\geq 1-\sum_{\ell=1}^{+\infty} p_-^{k_{\ell-1}^+}-\frac{p}{2},
\end{equation}
where $p_-$ is defined in Lemma~\ref{lem:exittimebound} (i).\\
(ii) Let $0<\hat\nu<1$, $0<p\leq 1$. Then there exists $\hat \delta>0$ so that, for all  $0<\delta\leq \hat \delta$ and $k_\ell^+=2^{\delta(\ell+\eta_0)}$, there exists $\tilde \eta$ such that, $\forall \eta_0\geq \tilde \eta$, 
\begin{align*}\label{eq:tauellub}
\{G_1, \forall \ell \in \N, \tau_\ell^{1-\hat\nu} \leq \frac12 2^{2\eta_\ell}\}&\supset G_1\cap G_2 \nonumber\\
\P\left(G_1, \forall \ell \in \N, \tau_\ell^{1-\hat\nu} \leq \frac122^{2\eta_\ell}\right)&\geq \P(G_1\cap G_2)\geq 1-p.
\end{align*}
(iii) $\forall 0\leq p<1$, $\forall \delta>0$, $\exists \tilde\eta>\eta_+$ so that for all $\eta_0\geq \tilde \eta$ and for all sequences $0<k_\ell^-\leq\ell$, one has
\begin{equation}\label{eq:dwelltimelb}
\P(G_1\cap G_3)\geq 1-\sum_{\ell=1}^{+\infty} q_+^{k_\ell^-}-\frac{p}{2},
\end{equation}
where $q_+$ is defined in Lemma~\ref{lem:exittimebound} (ii).\\
(iv) $\forall 0\leq p<1$, $\forall \delta>0$, $\exists \tilde\eta>\eta_+$ so that for all $\eta_0\geq \tilde \eta$
\begin{align*}
\{G_1, \forall \ell\in \N, \tau_\ell\geq a_\ell\}&\supset G_1\cap G_3\nonumber\\
\P(G_1, \forall \ell\in \N, \tau_\ell\geq a_\ell)\geq \P(G_1\cap G_3)&\geq 1-p,
\end{align*}
provided $k_\ell^-=\min\{\delta(\ell+\eta_0), \ell\}$ and $a_\ell$ is given by \eqref{eq:al}. 
\end{proposition}
We point out that, in order to get a sharp upper bound on the $\tau_\ell-\tau_{\ell-1}$ in part (i) of the lemma, one would like to take the $k^+_\ell$ small, or at least bounded, in the left hand side of \eqref{eq:dwelltimeub}. But this estimate is useful only if the $k^+_\ell$ are large for all $\ell$ and tend to $+\infty$ as $\ell\to+\infty$. This is indeed needed for the sum in the right hand side to converge to a small number. 
\begin{proof} (i) First note that it follows from Proposition~\ref{prop:jumpproba} (ii) that, for all $0\leq p<1$ and all $\delta>0$, there exists $\eta_*>\eta_+$ so that, for all $\eta_0\geq \eta_*$,
$
\P(G_1)\geq 1-\frac{p}{2}.
$
Hence 
\begin{equation}\label{eq:gg0}
\P(G_1\cap G_2)\geq 1-\frac{p}{2} - \P(G_1\cap G_2^c).
\end{equation}
Now,
\begin{equation}\label{eq:gg1}
\P(G_1\cap G_2^c) \leq  \sum_{\ell=1}^{+\infty}\P(G_2(\ell)^c\cap G_1) \leq  \sum_{\ell=1}^{+\infty}\P(G_2(\ell)^c\cap G_1(\ell-1))\leq  \sum_{\ell=1}^{+\infty}\P(G_2(\ell)^c| G_1(\ell-1))
\end{equation}
and, for all $\ell\in\N_*$,
\begin{align}\label{eq:gg2}
\P(G_2(\ell)^c|G_1(\ell-1))&\leq \sup_{\eta_{\ell-1}\in I_{\ell-1}} \P(G_2(\ell)^c|\eta_{\ell-1}\}
\end{align}
where we used the observation that $G_1(\ell-1)=\{\eta_{\ell-1}\in I_{\ell-1}\}$, where $$I_{\ell-1}=[\eta_0(1-\delta) +(\mu-\delta)(\ell-1),   \eta_0(1+\delta) + (\mu+\delta)(\ell-1)].$$
Now, proceeding as in the beginning of the proof of Proposition~\ref{prop:jumpproba}, 
\begin{align}
\P(G_2(\ell)^c|\eta_{\ell-1})&=\P(\tau_\ell-\tau_{\ell-1} > k_{\ell-1}^+2^{2\eta_{\ell-1}}\mid \eta_{\ell-1})\nonumber\\
&= \sum_i \P\left(\tau_\ell-\tau_{\ell-1} > k_{\ell-1}^+2^{2\eta_{\ell-1}}\mid \eta_{\ell-1}, \tau_{\ell-1}=i\right)\P(\tau_{\ell-1}=i\mid \eta_{\ell-1}).
\end{align}
We have, for all $i$, $\ell$ and $\eta_{\ell-1}$,
\begin{align*}
\P(\tau_\ell-\tau_{\ell-1} > &k_{\ell-1}^+2^{2\eta_{\ell-1}}\mid \eta_{\ell-1}, \tau_{\ell-1}=i)\\
&=
\P(\inf\{k|\xi_{i+k}\not\in]2^{\eta_{\ell-1}-1}+L, 2^{\eta_{\ell-1}+1}-L[\}> k_{\ell-1}^+2^{2\eta_{\ell-1}}\mid \xi_i\in J_{\eta_{\ell-1}}, \tau_{\ell-1}=i)\\
&\leq\sup_{\xi_i\in J_{\eta_{\ell-1}}}\P(\inf\{k|\xi_{i+k}\not\in]2^{\eta_{\ell-1}-1}+L, 2^{\eta_{\ell-1}+1}-L[\}> k_{\ell-1}^+2^{2\eta_{\ell-1}}\mid \xi_i)\\
&\leq \sup_{\xi_i\in J_{\eta_{\ell-1}}}\P(\inf\{k|\xi_{i+k}\not\in]2^{\eta_{\ell-1}-1}, 2^{\eta_{\ell-1}+1}[\}> k_{\ell-1}^+2^{2\eta_{\ell-1}}\mid \xi_i),
\end{align*}
where we used \eqref{eq:markovineq}. We now remark that
$
\inf\{k|\xi_{i+k}\not\in]2^{\eta_{\ell-1}-1}, 2^{\eta_{\ell-1}+1}[\}=K_i,
$
where $K_i$ is defined in \eqref{eq:Kn0}, with $b_-=\frac12$, $b_+=2$ and $r=2^{\eta_{\ell-1}}$. It therefore follows from Lemma~\ref{lem:exittimebound} (i) and from what precedes that, provided $2^{\tilde\eta}\geq \xi_*$, we have for all $\eta_{\ell-1}\geq \tilde\eta$,
$$
\P(\tau_\ell-\tau_{\ell-1} > k_{\ell-1}^+2^{2\eta_{\ell-1}}\mid \eta_{\ell-1}, \tau_{\ell-1}=i)
\leq
p_-^{k_{\ell-1}^+}.
$$
Using this in \eqref{eq:gg1}, we find that, for all $\tilde \eta\geq \log_2\xi_*$,
$$
\P(G_2(\ell)^c|\eta_{\ell-1})=\P(\tau_\ell-\tau_{\ell-1} > k_{\ell-1}^+2^{2\eta_{\ell-1}}| \eta_{\ell-1})\leq p_-^{k_{\ell-1}^+},
$$
which, when inserted into \eqref{eq:gg0}-\eqref{eq:gg2}, yields the result provided $\tilde\eta\geq \max\{\eta_*, \log_2\xi_*\}$.\\
%%%%%%%%%
(ii) Let $\hat\nu>0, 0\leq p<1$. Let $\delta>0$.
On $G_1\cap G_2$, a simple calculation using $k_\ell^+=2^{\delta(\ell+\eta_0)}$ yields
\begin{align*}
\tau_\ell=\sum_{k=1}^\ell \tau_k-\tau_{k-1}\leq (2^{2\mu}-1)^{-1}2^{2\left[(1+\frac32\delta)\eta_0 +(\mu +\frac32\delta)\ell\right]}.
\end{align*}
Introducing 
$$ \hat{\delta} = \left(5-3\hat{\nu}\right)^{-1}\min\left(2\mu\hat{\nu}, 2\hat{\nu}-\eta_+^{-1}\left((1-\hat{\nu})\log_2|2^{2\mu}-1|-1\right)\right),$$ 
it now easily follows that, if $\delta\leq\hat\delta$, and $\eta_0\geq \eta_+$, then, on $G_1\cap G_2$,
$
\tau_\ell^{1-\hat\nu}\leq \frac122^{2\eta_\ell},
$
which is the desired estimate. To see it occurs with high probability, we use \eqref{eq:dwelltimeub} to check that there exists $\tilde\eta>\eta_+$, depending on $\delta$ and $p$, so that, for all $\eta_0\geq \tilde\eta$, one has
\begin{equation*}
\P(G_1\cap G_2)\geq 1-p.
\end{equation*}
\\
(iii)  As in (i), we argue that, for all $0\leq p<1$ and all $\delta>0$, there exists $\eta_*>\eta_+$ so that, for all $\eta_0\geq \eta_*$,
\begin{equation}\label{eq:gg3}
\P(G_1\cap G_3)\geq 1-\frac{p}{2} - \sum_{\ell=1}^{+\infty}\P(G_3(\ell)^c\cap G_1(\ell-1)\cap\dots\cap G_1(\ell- k_\ell^--1)).
\end{equation}
For ease of notation, we introduce, for $\ell-k_\ell^-\leq k\leq \ell$, 
$$
G_4(k,\ell)=\{\tau_k-\tau_{k-1}< a_\ell\},\quad \mathrm{and}\quad G_5(k,\ell)=G_4(k, \ell)\cap G_1(k-1).
$$
Remarking that
$$
G_3(\ell)^c=\{\forall \ell-k_\ell^-\leq k\leq \ell, \tau_k-\tau_{k-1}< a_\ell\}=\cap_{k=l-k_\ell^-}^\ell G_4(k,\ell),
$$
and introducing
$$
G_5^-(k,\ell) =\cap_{k'=\ell-k^-_\ell}^{k-1} G_5(k',\ell),
$$
we can then write
\begin{align}
\P(G_3(\ell)^c\cap G_1(\ell-1),\dots, G_1(\ell- k_\ell^--1))&=\P(\cap_{k=\ell-k_\ell^-}^\ell G_5(k,\ell))\nonumber\\
&=\left[\prod_{k=0}^{k_\ell^--1}\ \P\left(G_5(\ell-k,\ell) | G_5^-(\ell-k,\ell)\right)\right]\P(G_5(\ell-k_\ell^-,\ell)).\label{eq:chainrule}
\end{align}
Now, for $\ell-k_\ell^-< k'\leq \ell$, we have
\begin{align}\label{eq:kk1}
\P(G_5(k', \ell) &| G_5^-(k',\ell))\nonumber\\
&=\P(G_4(k',\ell)|G_1(k'-1), G_5^-(k',\ell))\P(G_1(k'-1)|G_5^-(k',\ell))\nonumber\\
&\leq \P(G_4(k',\ell)|G_1(k'-1), G_5^-(k',\ell))\nonumber\\
&\leq \P(\Delta\tau_{k'}< a_\ell | \eta_{k'-1}\in I_{k'-1}, G_5^-(k',\ell))\nonumber\\
&\leq \sup_{\stackrel{\eta_{j}\in I_{j}}{\ell-k_\ell^--1\leq j< k'}}\P(\Delta\tau_{k'}< a_\ell | \eta_{k'-1}, \eta_{k'-2}, \dots \eta_{\ell-k_\ell^--1}, \Delta\tau_{{k'}-1}<a_\ell,\dots, \Delta\tau_{\ell-k_\ell^-}<a_\ell),
\end{align}
where we introduced $\Delta \tau_j=\tau_j-\tau_{j-1}$. As before, we write
\begin{align}
 \P&
(\Delta\tau_{k'}< a_\ell | \eta_{k'-1}, \eta_{k'-2}, \dots \eta_{\ell-k_\ell^--1}, \Delta\tau_{{k'}-1}<a_\ell,\dots, \Delta\tau_{\ell-k_\ell^-}<a_\ell)\nonumber\\
 &=\sum_{0<i_j-i_{j+1}<a_\ell}\P(\Delta\tau_{k'}< a_\ell | \eta_{k'-1}, \eta_{k'-2}, \dots \eta_{\ell-k_\ell^--1}, \tau_{{k'}-1}=i_{k'-1},\dots, \tau_{\ell-k_\ell^--1}=i_{\ell-k_\ell^--1})\nonumber\\
&\times\P(\tau_{{k'}-1}=i_{k'-1},\dots, \tau_{\ell-k_\ell^--1}=i_{\ell-k_\ell^--1}|\eta_{k'-1}, \eta_{k'-2}, \dots \eta_{\ell-k_\ell^--1},\Delta\tau_{{k'}-1}<a_\ell,\dots, \Delta\tau_{\ell-k_\ell^-}<a_\ell).
\end{align}
 It remains to estimate
\begin{align*}
\P(\Delta\tau_{k'}< a_\ell |& \eta_{k'-1}, \eta_{k'-2}, \dots \eta_{\ell-k_\ell^--1}, \tau_{{k'}-1}=i_{k'-1},\dots, \tau_{\ell-k_\ell^--1}=i_{\ell-k_\ell^--1}).
%\\
\end{align*}
For that purpose, we make the observation that
\begin{equation*}
\{\eta_{k'-1}, \eta_{k'-2}, \dots \eta_{\ell-k_\ell^--1}, \tau_{{k'}-1}=i_{k'-1},\dots, \tau_{\ell-k_\ell^--1}=i_{\ell-k_\ell^--1}\}=\{\xi_{i_{k'-1}}\in J_{\eta_{k'-1}}, C\}
\end{equation*}where
\begin{equation*}
C=\{\eta_{k'-2}, \dots \eta_{\ell-k_\ell^--1}, \tau_{{k'}-1}>i_{k'-1}-1,\dots, \tau_{\ell-k_\ell^--1}=i_{\ell-k_\ell^--1}\}.
\end{equation*}
Indeed, on the set where $\xi_{i_{k'-1}}\in J_{\eta_{k'-1}}$ and $\tau_{k'-1}>i_{k'-1}-1$,  we
do have $\tau_{k'-1}=i_{k'-1}$. Hence
\begin{align}
\P(\Delta\tau_{k'}< a_\ell &| \eta_{k'-1}, \eta_{k'-2}, \dots \eta_{\ell-k_\ell^--1}, \tau_{{k'}-1}=i_{k'-1},\dots, \tau_{\ell-k_\ell^--1}=i_{\ell-k_\ell^--1})\nonumber\\
&=\P(\inf\{t | \xi_{i_{k'-1}+t}\not \in ]2^{\eta_{i_{k'-1}}-1}+L, 2^{\eta_{i_{k'-1}}+1}-L[\}<a_\ell| \xi_{i_{k'-1}}\in J_{\eta_{k'-1}}, C)\nonumber\\
&\leq \sup_{\xi_{i_{k'-1}}\in J_{\eta_{k'-1}}}\P(\inf\{t | \xi_{i_{k'-1}+t}\not \in ]2^{\eta_{i_{k'-1}}-1}+L, 2^{\eta_{i_{k'-1}}+1}-L[\}<a_\ell| \xi_{i_{k'-1}})
\end{align}
where we use the observation that $C\in {\mathcal F}_{i_{k'-1}-1}$, and \eqref{eq:markovineq}. We now wish to use Lemma~\ref{lem:exittimebound} to conclude. For that purpose, first note that, there exist $0<b_-<1<b_+$ so that for all $r\geq 2^{\eta_+}$, 
$$
\frac12 r+L < b_- r<r<b_+r<2r-L.
$$
Clearly, one can think of $b_-$ as being close to $\frac12$ and of $b_+$ as being close to $2$. With the notation of \eqref{eq:Kn0}, and $r=2^{\eta_{i_{k'-1}}}$, one has
$$
K_{i_{k'-1}} < \inf\{t | \xi_{i_{k'-1}+t}\not \in ]2^{\eta_{i_{k'-1}}-1}+L, 2^{\eta_{i_{k'-1}}+1}-L[\},
$$
so that
$$
\P(\inf\{t | \xi_{i_{k'-1}+t}\not \in ]2^{\eta_{i_{k'-1}}-1}+L, 2^{\eta_{i_{k'-1}}+1}-L[\}<a_\ell| \xi_{i_{k'-1}})\leq \P(K_{i_{k'-1}}< a_\ell  | \xi_{i_{k'-1}}).
$$
Recalling that $\ell-k_\ell^-< k'\leq \ell$, one checks readily that
$$
a_\ell\leq r^2,
$$
so that \eqref{eq:Klowerbd} implies, there exists $\hat \eta_*$ so that, for all $\eta_0\geq \hat\eta_*$, for all $\eta_{k'-1}\in I_{k'-1}$ and for all $\xi_{i_{k'-1}}\in J_{\eta_{k'-1}}$,
\begin{align}\label{eq:kk2}
\P(\inf\{t | \xi_{i_{k'-1}+t}\not \in ]2^{\eta_{i_{k'-1}}-1}+L, 2^{\eta_{i_{k'-1}}+1}-L[\}<a_\ell| \xi_{i_{k'-1}})&\leq \P(K_{i_{k'-1}}< a_\ell  | \xi_{i_{k'-1}})\nonumber\\
&\leq \P(K_{i_{k'-1}}< r^2  | \xi_{i_{k'-1}})\nonumber\\
&\leq q_+<1,
\end{align}
provided $b'_-, b'_+$ are chosen so that 
$$
b_-r<b'_-r\leq r-L\leq r+L\leq b'+r<b_+r,
$$
for all $r\geq 2^{\eta_+}$, which is always possible. (One should think of $b'_{\pm}$ as being close to $1$.) Inserting \eqref{eq:kk1}-\eqref{eq:kk2} into \eqref{eq:chainrule} yields
$$
\P(G_3(\ell)^c\cap G_1(\ell-1),\dots, G_1(\ell- k_\ell^--1))\leq q_+^{k_\ell^-},
$$
provided $\eta_0\geq \tilde \eta=\max\{\eta_*, \hat \eta_*\}$. Inserting this in \eqref{eq:gg3} yields \eqref{eq:dwelltimelb}.\\
(iv) This is now an immediate consequence of (iii).
\end{proof}

%%%%%%%%%%%%%%%%%%%%%%%%%%%%%%%%%
%SECTION 7 Proof of Th3.1
%%%%%%%%%%%%%%%%%%%%%%%%%%%%%%%%%

\section{Proof of Theorem~\ref{thm:finalresult}}\label{s:finalproof}
(i) Let $p>1$ and $0<\nu<1$. Let $0<\hat\nu<\nu$. It then follows from Proposition~\ref{prop:dwelltimes} that, there exists $\hat\delta$ and $\tilde\eta$ so that, for all $0<\delta<\hat\delta$ and for all $\eta_0\geq \tilde\eta$, one has 
$$
\P(G)\geq 1-p,
$$
where $G=G_1\cap G_2\cap G_3$ and where $k_\ell^+=2^{2\delta(\eta_0+\ell)}$, $k_\ell^-=\min\{\delta(\eta_0+\ell), \ell\}$. Note that $G$ depends on $\eta_0$ and $\delta$. In addition, on $G$, the following inequalities hold for all $\ell\in \N_*$:
\begin{align*}
&|\eta_\ell-\mu\ell-\eta_0|\leq \delta(\ell+\eta_0),\\
&\tau_\ell^{1-\hat\nu}\leq \frac122^{2\eta_\ell},\\
&\tau_\ell\geq a_\ell=2^{2[(1-\delta)\eta_0+(\mu-\delta)(\ell-1-k_\ell^-)]}2^{-\delta k_\ell^-}.
\end{align*}
Since $(2^{\eta_0}+\tau_\ell^{\frac12})^{1-\hat\nu}\leq 2^{(1-\hat\nu)\eta_0}+ \tau_\ell^{\frac{1-\hat\nu}{2}}$, one easily infers from the first two inequalities that, on $G$, provided $\delta\leq \min\{\hat \delta, \hat\nu\}$, one has for all $\eta_0\geq \max\{\tilde\eta, (2(\hat\nu-\delta))^{-1}\}$, and for all $\ell$,
\begin{equation}\label{eq:f1}
\left(2^{\eta_0}+\tau_\ell^{\frac{1}{2}}\right)^{1-\hat\nu}\leq 2^{\eta_\ell}.
\end{equation}
Similarly, using the first and third inequality above, one shows that on $G$, provided $\delta\leq \min\{\hat \delta, \hat\nu\}$, one has for all $\eta_0\geq \max\{\tilde\eta, (2(\hat\nu-\delta))^{-1}\}$, %[[check and fix constants]] 
and for all $\ell$,
\begin{equation}\label{eq:f2}
 2^{\eta_\ell}\leq\left(2^{\eta_0}+\tau_\ell^{\frac{1}{2}}\right)^{1+\hat\nu}.
\end{equation}
%\begin{equation}\label{eq:ingetaell}
%\left(2^{\eta_0}+\tau_\ell^{\frac{1}{2}}\right)^{1-\hat\nu}\leq 2^{\eta_\ell}\leq \left(2^{\eta_0}+\tau_\ell^{\frac{1}{2}}\right)^{1+\hat\nu}.
%\end{equation}
We are now ready to conclude the proof. By the definition of the stopping time $\tau_\ell$, and using that $\eta_\ell\geq \eta_{\ell+1}-1$, as well as \eqref{eq:f1}-\eqref{eq:f2},we have for all $k\in [\tau_\ell;\tau_{\ell+1}[$
\begin{displaymath}
\frac14\left(2^{\eta_0}+k^{\frac12}\right)^{1-\hat\nu}\leq\frac142^{\eta_{\ell+1}}\leq 2^{\eta_\ell-1}+L\leq \xi_k\leq 2^{\eta_\ell+1}-L\leq 22^{\eta_\ell}\leq2\left(2^{\eta_0}+k^{\frac12}\right)^{1+\hat\nu}
\end{displaymath}
Finally, remarking that
$$
2=(2^{\eta_0})^{\frac1{\eta_0}}\leq (2^{\eta_0}+k^{\frac12})^{\frac1{\eta_0}}
$$
one obtains the result if one chooses $\tilde \eta$ large enough so that 
$$
\hat\nu+\frac2{\tilde\eta}\leq\nu.
$$
(ii) This is an immediate consequence of (i).

%%%%%%%%%%%%%%%%%%%%%%%%%%%%%%%%%
%APPENDIX
%%%%%%%%%%%%%%%%%%%%%%%%%%%%%%%%%
\appendix
\section{Appendix}\label{s:appendix}

In this Appendix we prove Theorem~\ref{thm:scaling} and Lemma~\ref{lem:exittimeBessel}. Recall that  we consider a family of continuous and piecewise linear stochastic processes $\left( t\to R^\eps(t), t\in [0,2T]\right)_{0<\eps<\eps_*}$ defined as follows. For each $0<\eps<\eps_*$, where $\eps_*\ll1$
$$R^\eps(t_n)=\eps\xi_n, \qquad t_n=n\eps^2,$$
$$R^\eps(t)=\dfrac{t_{n+1}-t}{\eps^2}R^\eps(t_n)+\dfrac{t-t_n}{\eps^2}R^\eps(t_{n+1}), \qquad t\in [t_n,t_{n+1}].$$
Here $\left( \xi_n\right)_n$ is defined in \eqref{eq:Mchain}.
Note that the initial value $R^\eps(0)=1$ is independent of $\eps$ and non-random. Each realization of the process $\left( R^\eps(t)\right)_{t\in [0,2T]}$ belongs to
$$ \mathcal{C}:= \left( \mathcal{C}\left( [0,2T]: \R_+\right), ||\cdot||_\infty\right).$$
Let $\mathcal{B}(\mathcal{C})$ designate the Borel sets of $\mathcal{C}$.

The method used in the proof of Theorem~\ref{thm:scaling} is standard. It is in particular described in \cite{GR09}. It can be decomposed in $3$ steps.
\begin{itemize}
\item[Step $1$]For each $\eta>1$, we introduce the process $X^\eps$ which is $R^\eps$ stopped at $\eta^{-1}$ or at $\eta$ (see \eqref{eq:taueps}). We show that the process $X^\eps$ admits convergent subsequences as $\eps\to 0$ by showing it is precompact.
\item[Step $2$] We show the limits of the converging subsequences are solutions of the martingale problem associated to a Bessel process in dimension $2\gamma+1$, stopped at $\eta^{-1}$ or $\eta$. As the latter is well-posed, we conclude that the limits have the distribution of the preceding stopped Bessel process and that it is not only the subsequences which are converging but the entire family. 
\item[Step $3$] We show that the convergence result still holds when we delete all the stopping times, which means we tak $\eta\to+\infty$. The transience of the Bessel process in dimension strictly larger than $2$ is an essential ingredient in this part of the proof.
\end{itemize}

\begin{proof}[Proof of Theorem~\ref{thm:scaling}]
$ $
\vspace*{-1em}
\paragraph{Step 1: Precompactness of the stopped proccess}
Let $\eta\gg 1$ and $\eps_*=(\eta\xi_+)^{-1}$, then for all $0<\eps<\eps_*$, $\eta^{-1}>\eps\xi_+$. We introduce for all $0<\eps<\eps_*$ the stopping time

\begin{equation}\label{eq:taueps}
\tau^\eps:=\inf \lbrace t\in [0,T];\, R^{\eps}(t)\not \in (\eta^{-1},\eta)\rbrace
\end{equation} 
with the convention $\inf\{\emptyset\}=2T$. We then introduce the stopped process
$$ \forall t\in [0,2T],\quad  X^\eps(t)=R^\eps(t\wedge \tau^\eps).$$
In other words, once $X^\eps$ reaches $\eta$ or $\eta^{-1}$ it stays constant. The assumption $\eps_*=(\eta\xi_+)^{-1}$ guarantees that for all $n$, $\xi_n=\frac{1}{\eps}R^{\eps}(t_n)>\xi_+$.

Introduce 
$$\mathcal{A}^\eps f(x)= a^\eps(x)f''(x)+b^\eps(x)f'(x)+c^\eps(f,x),$$
 where
\begin{align}
\label{eq:aeps}a^\eps(x)&=\eps^2+\eps^4\frac{\gamma^2}{x^2}+\eps^{2\alpha}\E\left( G_0^\eps(x,\omega)^2\right)+\eps^{2\beta}\E\left( G_1(x,\omega)^2\right)+\eps^{\beta+3}\E\left( G_1(x,\omega)\right),\\
\label{eq:beps}b^\eps(x)&= \eps^2\frac{\gamma}{x}+\eps^{\beta+1}\E\left(G_1(x,\omega)\right),\\
\label{eq:ceps}c^\eps(f,x)&=\sum_{n=3}^{K}f^{(n)}(x)\E\left( \left(\eps\omega+\eps^2\frac{\gamma}{x}+\eps^{\alpha+1}G_0^\eps(x,\omega)+\eps^{\beta+1}G_1^\eps(x,\omega) \right)^n\right)+ O\left(\|\eps f\|^{K+1}\right),
\end{align}
and 
$$\mathcal{D}_*^\eps=\lbrace f\in \mathcal{C}\left([\eta^{-1},\eta]\right)\cap\mathcal{C}^\infty\left( (\eta^{-1},\eta)\right),\, \lim_{x\to \eta^{\pm 1}}\mathcal{A}^\eps f(x)=0\rbrace.$$
Then we have the following lemma.
\begin{lemma}\label{lem:core}The operator $\left( \mathcal{A}^\eps, \mathcal{D}_*^\eps\right)$ is a core for the infinitesimal generator of the stopped proccess $\left( X^\eps(t_n)\right)_n$.
\end{lemma}
See \cite{Stroock} and~\cite{Mandl} for the proof.

Hence, as for all $f\in \mathcal{D}_*^\eps$, the process $\left( f(X^\eps(t_n))-\sum_{j=0}^{n-1}\mathcal{A}^\eps f(X^\eps(t_j))\right)_{n\in \N}$ is a martingale, it is easy to check by \eqref{eq:aeps}-\eqref{eq:ceps} and \eqref{eq:g0g1} that for all $f\in \mathcal{D}_*^\eps$ there exists a constant $0<C_f<+\infty$ depending only on $f$ such that the process
$\left(f\left( X^\eps(t_n)\right)-C_f t_n\right)_n$ is a sub-martingale.
As well, for all $\delta>0$ there exists $\tilde{\eps}$ such that for all $\eps<\tilde{\eps}$ we have 
$$ \P\left(\left\vert X^\eps(t_j)-X^\eps(t_{j-1})\right\vert >\delta\right)=0,$$ 
and then
\begin{equation}\label{eq:controlstep}
\lim_{\eps\to 0}\sum_{j=1}^{\lfloor \frac{2T}{\eps^2}\rfloor}\P\left(\left\vert X^\eps(t_j)-X^\eps(t_{j-1})\right\vert >\delta\right)=0,
\end{equation}
which assure by Theorem \textit{1.4.11} of \cite{Stroock}, the precompactness of the family $\left( t\to X^\eps(t),\, t\in [0,T]\right)_{0<\eps<\eps_*}$. This yields the existence of decreasing functions $\varphi:(0,\eps_*)\to (0,\eps_*)$ such that 
\begin{equation*}\label{eq:cvd}
\left( t\to X^{\varphi(\eps)}(t),\, t\in [0,2T]\right)_{0<\eps<\eps_*}\rightharpoonup \left( t\to X^{\varphi}(t),\, t\in [0,2T]\right),
\end{equation*}
where  the symbol $\rightharpoonup$ refers to convergence in distribution.

\paragraph{Step 2: Convergence and limit}

Introduce 
\begin{equation*}\label{eq:tauphi}
\tau^\varphi:=\left\lbrace t\in [0,2T]; \, X^\varphi(t)\not \in (\eta^{-1},\eta)\right\rbrace
\end{equation*}
with the convention $\inf\{\emptyset\}=2T$. 
As it evolves in the compact $[0,2T]$, the sequence $\left( \tau^\eps\right)_{0<\eps<\eps_*}$ is also tight, however the limit of a subsequence $(\tau^{\psi(\eps)})_{0<\eps<\eps_*}$ converging is not a stopping time.

\begin{theorem}\label{thm:weakly}
The processes $\left(X^\varphi(t\wedge \tau^\varphi)\right)_{t\in [0,2T}$ are solution of the martingale problem associated to the infinitesimal generator $\left( \mathcal{L}, \mathcal{D}_*\right)$ where 
$$\mathcal{L}:=\frac{1}{2}\dfrac{\rd^2}{\rd x^2}+\frac{\gamma}{x}\dfrac{\rd }{\rd x}$$
and 
$$\mathcal{D}_*:=\left \lbrace f\in \mathcal{C}\left( [\eta^{-1}, \eta]\right)\cap \mathcal{C}^\infty \left( (\eta^{-1},\eta)\right);\, \lim_{x\to \eta^{\pm 1}} \mathcal{L}f=0\right \rbrace.$$
\end{theorem}

Note that $\left( \mathcal{L}, \mathcal{C}^\infty(\R)\right)$ is the infinitesimal generator of a Bessel process of dimension $2\gamma +1$. The condition on $\mathcal{D}_*$: $\lim_{x\to \eta^{\pm 1}}\mathcal{L}f(x)=0$ yields that the Bessel process stays constant once it reaches the points $\eta^{\pm 1}$. We call the points $\eta^{\pm 1}$ as being adhesif (see \cite{Mandl}).  

We introduce $\left( R(t)\right)_{t\in [0,2T]}$ a Bessel process of dimension $2\gamma+1$ such that $R(0)=1$ and 
\begin{equation*}\label{eq:taubessel}
\tau:=\inf\left \lbrace t\in [0,2T]; \, R(t) \not\in (\eta^{-1},\eta) \right \rbrace
\end{equation*}
with the convention $\inf \{\emptyset\}=2T$. Then $\left( R(t\wedge \tau)\right)_{t\in [0,2T]}$ is generated by $\left( \mathcal{L}, \mathcal{D}_*\right)$ (see \cite{Mandl}).

As martingale problems associated to Bessel processes are well-posed, this theorem implies that all the $\left(X^\varphi(t\wedge \tau^\varphi)\right)_{t\in [0,2T]}$ have the same distribution which is the one of $\left( R(t\wedge \tau)\right)_{t\in [0,2T]}$. In particular $\tau$ and the $\tau^\varphi$ have the same distribution.

\begin{proof}[Proof of Theorem~\ref{thm:weakly}]

The process 
$$M^{\varphi(\eps)}_f(t_n)= f\left( X^{\varphi(\eps)}(\lfloor \frac{t_n\wedge \tau^\varphi}{\varphi(\eps)^2}\rfloor\varphi(\eps)^2)\right)-\sum_{j=1}^{\lfloor \frac{t_n\wedge \tau^\varphi}{\varphi(\eps)^2}\rfloor}\mathcal{A}^{\varphi(\eps)}f\left( X^{\varphi(\eps)}(t_j)\right)$$
 is a $\mathcal{F}_n^{X^{\varphi(\eps)}}$-martingale for all $f\in \mathcal{D}_*^{\varphi(\eps)}$. Nevertheless, for all $f\in \mathcal{D}_*$ it is only a submartingale. By Doob decomposition (see \cite{EthierKurtz}) we can write
\begin{equation}\label{eq:DoobforM}
M^{\varphi(\eps)}_f(t_n)=Mart^{\varphi(\eps)}_f(t_n)+O^{\varphi(\eps)}_f(t_n)
\end{equation}
where $\left(Mart^{\varphi(\eps)}_f(t_n)\right)_n$ is a $\mathcal{F}_n^{X^{\varphi(\eps)}}$-martingale and $\left( O^{\varphi(\eps)}_f(t_n)\right)_n$ is deterministic and tend to $0$ as $\eps \to 0$.

Then, applying the Representation Theorem of Skorohod (see~\cite{Billingsley}), there exist a probability space $\left( \tilde{\Omega},\tilde{\mathcal{F}},\tilde{\P}\right)$ and processes
$$\left( t\to \tilde{X}^{\varphi(\eps)}(t), t\in [0,2T]\right)_{0<\eps<\eps_*}, \quad \text{and } \left( \tilde{X}^{\varphi}(t)\right)_{t\in [0,2T]} $$
 respectively of same distribution than $\left( t\to {X}^{\varphi(\eps)}(t), t\in [0,2T]\right)_{0<\eps<\eps_*}$ and $\left( {X}^{\varphi}(t)\right)_{t\in [0,2T]}$, there exists also $\tilde{\tau}^{\varphi}$ stopping time for $\tilde{X}^{\varphi}$ with the same distribution than $\tau^\varphi$ and such that
 \begin{equation*}\label{eq:cvpstilde}
 \sup_{t\in [0,2T]}\left\vert \tilde{X}^\varphi(t\wedge \tilde{\tau}^\varphi)-\tilde{X}^{\varphi(\eps)}(t\wedge \tilde{\tau}^\varphi)\right\vert \xrightarrow[\eps\to 0]{}0, \qquad \tilde{\P}\text{-a.s.}
 \end{equation*}
  \begin{lemma}\label{lem:tilde}
\begin{itemize}
\item[i)]Let $0\leq t_1<t_2\leq 2T$, then for all $f\in \mathcal{D}_*$ we have
\begin{align*}
\lim_{\eps\to 0} &\,  f\left(\tilde{X}^{\varphi(\eps)}\left(\lfloor \frac{t_2\wedge \tilde{\tau}^\varphi}{\varphi(\eps)^2} \rfloor\varphi(\eps)^2 \right)\right) -f\left(\tilde{X}^{\varphi(\eps)}\left(\lfloor \frac{t_1\wedge \tilde{\tau}^\varphi}{\varphi(\eps)^2} \rfloor\varphi(\eps)^2 \right) \right)-\sum_{j=\lfloor \frac{t_1\wedge \tilde{\tau}^\varphi}{\varphi(\eps)^2} \rfloor}^{\lfloor \frac{t_2\wedge \tilde{\tau}^\varphi}{\varphi(\eps)^2} \rfloor-1}\mathcal{A}^{\varphi(\eps)}f\left(\tilde{X}^{\varphi(\eps)}(t_j)\right)\\
&= f\left(\tilde{X}^\varphi(t_2\wedge \tilde{\tau}^\varphi)\right)-f\left(\tilde{X}^{\varphi}(t_1\wedge \tilde{\tau}^\varphi)\right)-\int_{t_1\wedge \tilde{\tau}^\varphi}^{t_2\wedge \tilde{\tau}^\varphi}\mathcal{L}f\left( \tilde{X}^\varphi(s)\right)ds, \qquad \tilde{\P}\text{-a.s.}
\end{align*}
\item[ii)]The limit process $\left(\tilde{M}^\varphi_f(t)\right)_t$,
$$\tilde{M}^\varphi_f(t)=f\left( \tilde{X}^\varphi(t\wedge \tilde{\tau}^\varphi)\right)-\int_{t_1\wedge \tilde{\tau}^\varphi}^{t_2\wedge \tilde{\tau}^\varphi} \mathcal{L}f \left(\tilde{X}^\varphi(s) \right)ds$$
is a $\mathcal{F}_t^{\tilde{X}^\varphi}$-martingale.
\end{itemize} 
 \end{lemma}
\begin{proof}
\begin{itemize}
\item[i)]$\tilde{X}^\varphi$ is time-continuous as limit of $\tilde{X}^{\varphi(\eps)}$ which are time-continuous. Then, for all $f\in \mathcal{D}_*$
\begin{align*}
\left \vert f\left( \tilde{X}^\varphi(t\wedge \tilde{\tau}^\varphi)\right)-f\left( \tilde{X}^{\varphi(\eps)}\left( \lfloor \dfrac{t\wedge \tilde{\tau}^\varphi}{\varphi(\eps)^2}\rfloor\right)\right)\right\vert \leq & ||f'||_{\infty}\Big( \left \vert \tilde{X}^\varphi\left(t\wedge \tilde{\tau}\right)- \tilde{X}^\varphi\left(\lfloor \dfrac{t\wedge \tilde{\tau}^\varphi}{\varphi(\eps)^2}\rfloor\right)\right \vert \\
& + \left \vert \tilde{X}^\varphi\left( \lfloor \dfrac{t\wedge \tilde{\tau}^\varphi}{\varphi(\eps)^2}\rfloor\right)-\tilde{X}^{\varphi(\eps)}\left( \lfloor \dfrac{t\wedge \tilde{\tau}^\varphi}{\varphi(\eps)^2}\rfloor\right)\right \vert \Big)\\
&\xrightarrow[\eps \to 0]{} 0 \qquad \tilde{\P}-\text{a.s}.
\end{align*}
We have, now to control $\int_{t_1\wedge \tilde{\tau}^{\varphi}}^{t_2\wedge \tilde{\tau}^\varphi}\mathcal{L}f\left( \tilde{X}^\varphi(s)\right)\rd s -\sum_{j=\lfloor \frac{t_1\wedge \tilde{\tau}^\varphi}{\varphi(\eps)^2}\rfloor}^{\lfloor \frac{t_2\wedge \tilde{\tau}^\varphi}{\varphi(\eps)^2}\rfloor-1}\mathcal{A}^{\varphi(\eps)}f\left( \tilde{X}^{\varphi(\eps)}(t_j)\right)$ for all $f\in \mathcal{D}_*$.
\begin{align*}
\Big  \vert \int_{t_1\wedge \tilde{\tau}^{\varphi}}^{t_2\wedge \tilde{\tau}^\varphi}\mathcal{L}f\left( \tilde{X}^\varphi(s)\right)\rd s &-\sum_{j=\lfloor \frac{t_1\wedge \tilde{\tau}^\varphi}{\varphi(\eps)^2}\rfloor}^{\lfloor \frac{t_2\wedge \tilde{\tau}^\varphi}{\varphi(\eps)^2}\rfloor-1}\mathcal{A}^{\varphi(\eps)}f\left( \tilde{X}^{\varphi(\eps)}(t_j)\right)\Big \vert \\
\leq & \left \vert \int_{t_1\wedge \tilde{\tau}^{\varphi}}^{t_2\wedge \tilde{\tau}^\varphi}\mathcal{L}f\left( \tilde{X}^\varphi(s)\right)\rd s -\varphi(\eps)^2\sum_{j=\lfloor \frac{t_1\wedge \tilde{\tau}^\varphi}{\varphi(\eps)^2}\rfloor}^{\lfloor \frac{t_2\wedge \tilde{\tau}^\varphi}{\varphi(\eps)^2}\rfloor-1}\mathcal{L}f\left( \tilde{X}^{\varphi}(t_j)\right)\right \vert \\
+\varphi(\eps)^2 & \left \vert \sum_{j=\lfloor \frac{t_1\wedge \tilde{\tau}^\varphi}{\varphi(\eps)^2}\rfloor}^{\lfloor \frac{t_2\wedge \tilde{\tau}^\varphi}{\varphi(\eps)^2}\rfloor-1}\mathcal{L}f\left( \tilde{X}^{\varphi}(t_j)\right)-\sum_{j=\lfloor \frac{t_1\wedge \tilde{\tau}^\varphi}{\varphi(\eps)^2}\rfloor}^{\lfloor \frac{t_2\wedge \tilde{\tau}^\varphi}{\varphi(\eps)^2}\rfloor-1}\mathcal{L}f\left( \tilde{X}^{\varphi(\eps)}(t_j)\right) \right \vert \\
+ \varphi(\eps)^2 & \left \vert \sum_{j=\lfloor \frac{t_1\wedge \tilde{\tau}^\varphi}{\varphi(\eps)^2}\rfloor}^{\lfloor \frac{t_2\wedge \tilde{\tau}^\varphi}{\varphi(\eps)^2}\rfloor-1}\mathcal{L}f\left( \tilde{X}^{\varphi(\eps)}(t_j)\right)- \frac{1}{\varphi(\eps)^2}\sum_{j=\lfloor \frac{t_1\wedge \tilde{\tau}^\varphi}{\varphi(\eps)^2}\rfloor}^{\lfloor \frac{t_2\wedge \tilde{\tau}^\varphi}{\varphi(\eps)^2}\rfloor-1}\mathcal{A}^{\varphi(\eps)}f\left( \tilde{X}^{\varphi(\eps)}(t_j)\right) \right \vert.
\end{align*}
The first term tends to $0$ when $\eps \to 0$ as an approximation of the integral. The second term tends to $0$ as $\eps \to 0$ too because of the convergence almost sure of $\tilde{X}^{\varphi(\eps)}$ to $\tilde{X}^\varphi$. For the third term, we use \cite{Stroock} and show that for each $f\in \mathcal{D}_*$,  
$$\frac{1}{\varphi(\eps)^2}\mathcal{A}^{\varphi(\eps)}f\xrightarrow[\eps\to 0]{}\mathcal{L}f$$
 uniformly on the compact subsets of $[\eta^{-1},\eta]$.

\item[ii)]We have to show that for all $0\leq s<t \leq 2T$ and all $0\leq s_1<\cdots<s_d\leq s$, we have, for all application $\phi\in \mathcal{C}^\infty_c\left( \R^d\right)$,
$$\E\left( \left(\tilde{M}^\varphi_f(t\wedge \tilde{\tau}^\varphi)-\tilde{M}^\varphi_f(s\wedge \tilde{\tau}^{\varphi})\right) \phi\left(\tilde{X}^\varphi(s_1\wedge \tilde{\tau}^\varphi),\cdots, \tilde{X}^\varphi(s_d\wedge \tilde{\tau}^\varphi)\right)\right)=0.$$
For that we use point \textit{ii)} of Lemma~\ref{lem:tilde}, \eqref{eq:DoobforM} and conclude with the Dominated convergence Theorem.
\end{itemize}
\end{proof}
Then Theorem~\ref{thm:weakly} is obtain by returning to $\left( \Omega, \mathcal{F}, \P\right)$.

By Theorem~\ref{thm:weakly} the two stopping time $\tau:= \inf \{t\in [0,T]; \, R(t)\not \in (\eta^{-1},\eta)\}$ and $\tau^\varphi$ have the same distribution, we'll write $\tau$ for both.

We end this step with the following corollary.
\begin{corollary}
The family of processes $\left( t\to X^\eps(t\wedge \tau), \, t\in [0,2T]\right)_{0<\eps<\eps_*}$ converge weakly to $\left( R(t\wedge \tau)\right)_{t\in [0,2T]}$.
\end{corollary}
\begin{proof}
We have to check that for all $\phi:\mathcal{C}\left( [0,2T]:\R_+\right)\to \R$ continuous and bounded,
$$\lim_{\eps\to 0}\E\left(\phi\left( X^{\varphi(\eps)}(t\wedge \tau)\right) \right)=\E\left( R(t\wedge \tau)\right).$$
For that we use the tightness of $\left( t\to X^{\varphi(\eps)}(t\wedge \tau), t\in [0,2T]\right)_{0<\eps<\eps_*}$ in a reductio ad absurdum.
\end{proof}
 \end{proof}

\paragraph{Step 3: Suppression of the stopping times}
For the moment, we have the following weak convergence
$$\left( t\to R^\eps(t\wedge\tau^\eps \wedge \tau),\, t\in[0,2T]\right)_{0<\eps<\eps_*}\rightharpoonup \left( R(t\wedge \tau)\right)_{t\in [0,2T]}.$$

In this last step, we want to delete all the stopping times in order to expand the convergence to the whole family $\left( R^\eps\right)_\eps$. We will use the transience of the Bessel process for $d>2$ (here $d$ refers to the dimension of the Bessel process) and then deleting the stopping times remains to make $\eta \to +\infty$.

\begin{proposition}\label{prop:allcv}Let $\gamma>\frac{1}{2}$, then $\left( t\to R^\eps(t),\, t\in [0,T]\right)_{0<\eps<\eps_*}\rightharpoonup \left( R(t)\right)_{t\in [0,T]}$.
\end{proposition}
\begin{proof}
For $d>2$ (which means $\gamma>\frac{1}{2}$), the transience of the Bessel process $\left( R(t)\right)_{t\in [0,2T]}$ yields that
$$\P\left( \forall t\in [0,2T];\, 0<R(t)<+\infty\right)=1,$$
and then, $\lim_{\eta\to +\infty}\P\left( \tau=2T\right)=1$. Let a decreasing subsequence $\varphi:(0,\eps_*)\to (0,\eps_*)$ such that $\left(\tau^{\varphi(\eps)}\right)_{0<\eps<\eps_*}$ converge weakly to $\tau^{\varphi(0)}$ (which is not a stopping time), it's easy to show that $\P\left(\tau\leq \tau^{\varphi(0)}\right)=1$ and then 
$$\lim_{\eta\to +\infty}\P\left( \tau^{\varphi(0)}\leq 2T\right)=1.$$

As $\tau^{\varphi(0)}=\lim_{\eps\to 0}\tau^{\varphi(\eps)}$, we can deduce that for all $p>0$, there exist $\eta_c\gg 1$ and $\eps(\eta_c)>0$ such that for all $\eta>\eta_c$ and $\eps<\eps(\eta_c)$, we have $\P\left( \tau^{\varphi(\eps)}\geq T\right)>1-p$. From that, we can deduce that for all $\phi:\mathcal{C}\left( [0,T]:\R_+\right) \to \R$
$$ \left \vert \E \left( \phi\left( R^{\varphi(\eps)}(\cdot)\right)-\phi\left( R(\cdot)\right)\right) \right\vert\leq 2||\phi||_\infty p+ \left \vert \E\left( \phi\left(R^{\varphi(\eps)}(\cdot\wedge \tau^{\varphi(\eps)}\wedge \tau)\right)-\phi\left(R(\cdot\wedge \tau) \right)\right)\right\vert.$$
But by Theorem~\ref{thm:weakly}
$$\E\left( \phi\left(R^{\varphi(\eps)}(\cdot\wedge \tau^{\varphi(\eps)}\wedge \tau)\right)-\phi\left(R(\cdot\wedge \tau) \right)\right)\mapsto_{\eps\to 0}0.$$
Then, there exists $\eta_c\gg 1$ such that for all $\eta> \eta_c$, for all $\phi:\mathcal{C}\left( [0,T]:\R_+\right) \to \R$, 
\begin{equation}\label{eq:limpto0}
\lim_{p\to 0}\lim_{\eps\to 0}  \E \left( \phi\left( R^{\varphi(\eps)}(\cdot)\right)-\phi\left( R(\cdot)\right)\right)=0.
\end{equation}

Now, assume that $R^\eps$ doesn't converge weakly to $R$, then there exist $\phi:\mathcal{C}\left( [0,T]:\R_+\right) \to \R$ and a decreasing subsequence $\psi:(0,\eps_*)\to (0,\eps_*)$ and there exists $\delta>0$ such that 
$$ \left \vert \E\left( \phi\left( R^{\psi(\eps)}(\cdot)\right) -\phi\left( R(\cdot)\right)\right) \right\vert>\delta.$$
As $\tau^{\psi(\eps)}$ is tight, it admits converging subsequences $\tau^{\varphi(\psi(\eps))}$ and \eqref{eq:limpto0} implies that 
$$\E\left( \phi\left( R^{\phi(\psi(\eps))}(\cdot)\right) -\phi\left( R(\cdot)\right)\right)\to_{\eps\to 0} 0,$$
which is a contradiction and then ends the proof of both Proposition~\ref{prop:allcv} and Theorem~\ref{thm:scaling}.
\end{proof}

\end{proof}

We remark that the condition $\gamma>\frac{1}{2}$ just appears in the last step, in order that the Bessel process $\left( R(t)\right)_{t\in [0,T]}$ does not reaches $0$ or explodes in finite time. If $\gamma\leq \frac{1}{2}$ then the Bessel process is reccurent and we are not able to delete the stopping times with this method.

In the proof of Theorem~\ref{thm:finalresult}, we need some estimation of exit time for a transient Bessel process. In particular, we need to know the probability for a Bessel procces starting at $1$ to reach $2$ before $\frac{1}{2}$.

\textbf{Lemma $\mathbf{4.2}$\, }\textit{ 
Let $\gamma>-\frac12$, and let $R$ be a Bessel process of dimension $2\gamma+1$ with $R(0)=1$. Let, for $a_-<1<a_+$, 
$$
T_{a_-,a_+}=\inf\{t\geq 0 \mid R(t)\not \in ]a_-, a_+[\},\quad T_{a_-}=\inf\{t\geq 0 \mid R(t)<a_-\}, \quad T_{a_+}=\inf\{t\geq 0 \mid R(t)> a_+\}.
$$
(i) Then, for all $T\geq 0$,
\begin{equation}\label{eq:ABesselexit}
0<\P(T_{a_-, a_+}>T)<1.
\end{equation}
(ii) If in addition $\gamma>\frac12$, 
\begin{equation*}\label{eq:Besseltransient}
\P(T_{a_-}>T_{a_+})=\dfrac{a_-^{1-2\gamma}-1}{a_-^{1-2\gamma}-a_+^{1-2\gamma}}
\end{equation*}
}
\begin{proof} (i) See~\cite{Mandl},~\cite{RevuzYor} or also~\cite{EthierKurtz}.\\
(ii) This is readily shown using  the Optional Stopping Theorem, as follows.  Consider the process $M(t)= R(t)^{1-2\gamma}$. Note that $1-2\gamma\leq 0$ for $\gamma\geq \frac{1}{2}$ but since for these values of $\gamma$ the Bessel process is almost surely positive (see~\cite{RevuzYor}), $M(t)$ is well defined.  Considering
\begin{displaymath}
T_+^M=\inf\left\{t\geq 0 | M(t)=a_+^{1-2\gamma}\right\} \quad \mbox{and}\quad T_-^M=\inf\left\{t\geq 0| M(t)=a_-^{1-2\gamma}\right\}.
\end{displaymath}
it is clear that $\P\left(T_{a_+}<T_{a_-}\right)=\P\left(T_+^M<T_-^M\right)$. By the Ito lemma, it is easily checked this is a local martingale. It then follows from the Optional Stopping Theorem that
\begin{equation}\label{eq:optionalstopping}
\E\left(M\left(T_+^M\wedge T_-^M\right)\right)=1.
\end{equation}
On the other hand,
\begin{equation}\label{eq:espM}
\E\left(M\left(T_+^M\wedge T_-^M\right)\right)=a_+^{1-2\gamma}\P\left(T_+^M<T_-^M\right)+a_-^{1-2\gamma}\left(1-\P\left(T_+^M<T_-^M\right)\right).
\end{equation}
From \eqref{eq:optionalstopping} and \eqref{eq:espM}, we then obtain
$$
\P\left(T_+^M<T_-^M\right)=\dfrac{a_-^{1-2\gamma}-1}{a_-^{1-2\gamma}-a_+^{1-2\gamma}},
$$
which yields the desired result.
\end{proof}
\bibliographystyle{alpha}
\bibliography{biblio}

\end{document}